\documentclass[12pt, twoside]{article}
\usepackage{amsmath,amsthm,amssymb}
\usepackage{times}
\usepackage{enumerate}
\usepackage{bm}
\usepackage[all]{xy}
\usepackage{mathrsfs}
\usepackage{amscd}
\usepackage{color}
\usepackage{calc}
\def\titlerunning#1{\gdef\titrun{#1}}
\makeatletter
\def\author#1{\gdef\autrun{\def\and{\unskip, }#1}\gdef\@author{#1}}
\def\address#1{{\def\and{\\\hspace*{18pt}}\renewcommand{\thefootnote}{}%
		\footnote {#1}}
	\markboth{\autrun}{\titrun}}
\makeatother
\def\email#1{e-mail: #1}

\def\keywords#1{\par\medskip
	\noindent\textbf{Keywords.} #1}

\newtheorem{theorem}{Theorem}[section]
\newtheorem{corollary}[theorem]{Corollary}
\newtheorem{lemma}[theorem]{Lemma}
\newtheorem{proposition}[theorem]{Proposition}
\theoremstyle{definition}
\newtheorem{definition}[theorem]{Definition}
\newtheorem{remark}[theorem]{Remark}
\newtheorem{conjecture}[theorem]{Conjecutre}
\numberwithin{equation}{section}

\xyoption{all}

\def \N {\mathbb{N}}
\def \C {\mathbb{C}}

\def \a {\alpha }
\def \b {\beta}

\def \de {\delta}
\def \De {\Delta}
\def \la {\lambda}
\def \La {\Lambda}
\def\w {\omega}
\def\Om{\Omega}

\def\pa{\partial}
\def\na {\nabla}
\def\Ga{\Gamma}

\begin{document}
	\baselineskip=17pt
	
	\titlerunning{Hodge theory of holomorphic vector bundle on compact K\"{a}hler hyperbolic manifold}
	\title{Hodge theory of holomorphic vector bundle on compact K\"{a}hler hyperbolic manifold}
	
	\author{Teng Huang}
	
	\date{}
	
	\maketitle
	
	\address{T. Huang: School of Mathematical Sciences, University of Science and Technology of China; CAS Key Laboratory of Wu Wen-Tsun Mathematics, University of Science and Technology of China, Hefei, Anhui, 230026, P. R. China; \email{htmath@ustc.edu.cn;htustc@gmail.com}}
	
	\begin{abstract}
		Let $E$ be a holomorphic vector bundle over a compact K\"{a}hler manifold $(X,\w)$ with negative sectional curvature $sec\leq -K<0$, $D_{E}$ be the Chern connection on $E$. In this article we show that if $C:=|[\La,i\Theta(E)]|\leq c_{n}K$, then $(X,E)$ satisfy a family of Chern number inequalities. The main idea in our proof is study the $L^{2}$ $\bar{\pa}_{\tilde{E}}$-harmonic forms on lifting bundle $\tilde{E}$ over the universal covering space $\tilde{X}$. We also observe that there is a closely relationship between the eigenvalue of the Laplace-Beltrami operator $\De_{\bar{\pa}_{\tilde{E}}}$ and the Euler characteristic of $X$. Precisely, if there is a line bundle $L$ on $X$ such that $\chi^{p}(X,L^{\otimes m})$ is not constant for some integers $p\in[0,n]$, then the Euler characteristic of $X$ satisfies
		$(-1)^{n}\chi(X)\geq (n+1)+\lfloor\frac{c_{n}K}{2nC} \rfloor$.
	\end{abstract}
	\keywords{Hodge theory; Holomorphic vector bundle; K\"{a}hler hyperbolic; Chern number inequality}
	\section{Introduction}
	Let us start the article by recalling a Hopf conjecture related to the negativity of Riemannian sectional curvature.
	\begin{conjecture}\label{Co1}
		The Euler characteristic $\chi(X)$ of a compact 2n-dimensional Riemannian manifold $X$ with sectional curvature $K<0$ (resp. $K\leq 0$) satisfies $(-1)^{n}\chi(X)>0$ (resp.  $(-1)^{n}\chi(X)\geq0$).	
	\end{conjecture}
This is true for $n=1$ and $2$ as the Gauss–Bonnet integrands in these two low dimensional cases have the desired sign \cite{Chern}. However, in higher dimensions, it is known that the sign of the sectional curvature does not determine the sign of the Gauss-Bonnet-Chern integrand \cite{Ger}. The conjecture is still open in its full generality for $n\geq3$.  Therefore, Dodziuk \cite{Dod79} and Singer \cite{Sin} suggested to use $L^{2}$-cohomology to approach this
	problem as follows: Show $\mathcal{H}^{k}_{(2)}(X)=\{0\}$ for $k\neq n$ which implies the $L^{2}$-Betti number $b_{(2)}^{k}(X)=0$ for $k\neq n$ and $\mathcal{H}_{(2)}^{n}(X)\neq\{0\}$ which implies $b^{n}_{(2)}(X)\neq0$. However, Anderson \cite{And85} constructed simply connected complete negatively curved Riemannian
	manifolds on which this does not hold, thus indicating a certainly difficulty with this
	approach. The program outlined above was carried out by Gromov \cite{Gro} when the manifold in question is K\"{a}hler and is homotopy equivalent to a closed manifold with strictly negative sectional curvatures.  The main theorem in \cite{Gro} states that for a K\"{a}hler hyperbolic manifold $X$, $\mathcal{H}_{(2)}^{p,q}(\tilde{X})=\{0\}$ if and only if $p+q\neq\dim_{\C}X$, where $\mathcal{H}_{(2)}^{p,q}(\tilde{X})$ denotes the space of $L^{2}$-harmonic forms of type $(p,q)$ on $\tilde{X}$. The vanishing of $\mathcal{H}_{(2)}^{p,q}(\tilde{X})$ for $p+q\neq\dim_{\C}X$ is a consequence of the strong $L^{2}$-Lefschetz theorem. Nonvanishing for $p+q=\dim_{\C}X$ follows from the $L^{2}$-index theorem and an upper bound for the bottom of the spectrum, whose proof is based on a twisting (by tensoring $\bar{\pa}+\bar{\pa}^{\ast}$ with a line bundle equipped with a connection) trick due to Vafa and Witten.  By extending Gromov's idea and notion above to the nonnegative version, Jost-Zuo \cite{JZ} and Cao-Xavier \cite{CX} independently introduced the concept of K\"{a}hler parabolic and consequently settled Conjecture \ref{Co1} in the case of $K\leq 0$ for K\"{a}hler manifolds. The study of the $L^{2}$-harmonic forms on a complete Riemannian manifold is a very  fascinating and important subject. It also has numerous applications in the field of Mathematical Physics (see \cite{Hit}). Other results on $L^{2}$ cohomology can be found in \cite{And,Chen,Dod,McN1,McN2}.
	
	In this article, we consider the Hodge theory on a Hermitian vector bundle $E$ over a complete, K\"{a}hler manifold $X$, $\dim_{\C}X=n$, with a K\"{a}hler form $\w$. Define a smooth K\"{a}hler metric, $g(\cdot,\cdot)=\w(\cdot,J\cdot)$ on $X$, where $J$ is the complex structure on $X$. Let $d_{A}$ be a Hermitian connection on $E$. The formal adjoint operator of $d_{A}$ acting on $\Om^{k}(X,E):=\Om^{k}(X)\otimes E$  is $d^{\ast}_{A}=-\ast d_{A}\ast$, where $\Om^{k}(X)$ is smooth $k$-forms on $X$ and $\ast$ is the Hodge star operator with respect to the metric $g$. We denote by $\mathcal{H}^{k}_{(2)}(X,E)$ the space of $L^{2}$ harmonic forms in  $\Om^{k}(X)$ with respect to the Laplace-Beltrami operator $\De_{A}:=d_{A}d_{A}^{\ast}+d_{A}^{\ast}d_{A}$.
	
	A differential form $\a$ on a Riemannian manifold $(X,\rm{g})$ is bounded with respect to the Riemannian metric $\rm{g}$ if the $L^{\infty}$-norm of $\a$ is finite,
	$$\|\a\|_{L^{\infty}(X)}:=\sup_{x\in X}\|\a(x)\|_{\rm{g}}<\infty.$$
	We say that $\a$ is $d$(bounded) if $\a$ is the exterior differential of a bounded form $\b$, i.e., $\a=d\b$ and $\|\b\|_{L^{\infty}(X)}<\infty$.
	
	If $\w$ is $d$(bounded), the author in \cite{Hua}  extended the vanishing theorem of Gromov's to holomorphic vector bundle case. We denote by $\mathcal{A}^{1,1}_{E}$ the space of all integrable connections $d_{A}$, i.e., $F^{2,0}_{A}=F_{A}^{0,2}=0$. The important observation is that if the Hermitian connection $d_{A}\in\mathcal{A}^{1,1}_{E}$, then the operator $L^{k}$ could commutes with $\De_{A}$ for any $k\in\mathbb{N}^{+}$. Following the idea in \cite{Gro}, the author proved a vanishing theorem on the spaces $\mathcal{H}^{k}_{(2)}(X,E)$. Suppose that $E$ is a holomorphic Hermitian vector bundle on $X$.  We denote by $D_{E}=\pa_{E}+\bar{\pa}_{E}$ its Chern connection, i.e., $\bar{\pa}_{E}=\bar{\pa}$, by $D^{\ast}_{E}$ the formal adjoint of $D_{E}$ and by $\pa_{E}^{\ast}$, $\bar{\pa}_{E}^{\ast}$ the components of $D_{E}^{\ast}$ of type $(-1,0)$ and $(0,-1)$. Let $\Theta(E)=\pa_{E}\bar{\pa}_{E}+\bar{\pa}_{E}\pa_{E}$ be the curvature operator on $E$. It is clear that $\bar{\pa}^{2}=0$. Therefore, for any integer $p=0,1,\ldots,n$, we get a complex
	$$\Om^{p,0}(X,E)\xrightarrow{\bar{\pa}}\ldots\xrightarrow{\bar{\pa}}\Om^{p,q}(X,E)\xrightarrow{\bar{\pa}}\Om^{p,q+1}(X,E)\rightarrow\ldots,$$
	known as the Dolbeault complex of $(p,\bullet)$-forms with values in $E$. We can define two operators: $$\De_{\bar{\pa}_{E}}:=\bar{\pa}_{E}\bar{\pa}_{E}^{\ast}+\bar{\pa}_{E}^{\ast}\bar{\pa}_{E},\ \De_{\pa_{E}}:=\pa_{E}\pa_{E}^{\ast}+\pa_{E}^{\ast}\pa_{E}.$$
	Let us introduce, See \cite[Charp V]{Dem}
	$$\mathcal{H}^{p,q}_{(2);\bar{\pa}_{E}}(X,E):=\{\a\in\Om^{p,q}_{(2)}(X,E):\De_{\bar{\pa}_{E}}\a=0\}.$$
	There are many vanishing theorems for Hermitian vector bundles over a compact complex manifolds. All these theorems are based on  a priori inequality for $(p,q)$-forms with values in a vector bundle, known as the Bochner-Kodaira-Nakano inequality. This inequality naturally leads to several positivity notions for the curvature of a vector bundle (\cite{Gri69,Kod53a,Kod53b,Kod54,Nak55,Nak73}). 
	
	The first purpose of this paper is to study the Hodge theory of the holomorphic bundle $E$ on the compact K\"{a}hler manifold $X$ with negative sectional curvature. At first, we denote by 
	$$C=\max_{p,q}|C_{p,q}|:=|[\La,i\Theta(E)]|$$
	the operator norm of $[\La,i\Theta(E)]$, where
	$$C_{p,q}:=\sup_{\a\in\Om^{p,q}(X,E)\backslash\{0\}}\frac{|\langle[\La,i\Theta(E)]\a,\a\rangle_{L^{2}(X)}|}{\|\a\|^{2}}.$$
	We then have
	\begin{theorem}[=Proposition \ref{P4} and Theorem \ref{T8}]\label{T1}
		Let $(X,\w)$ be a compact K\"{a}hler manifold with sectional curvature bounded from above by a negative constant, i.e.,
		$$sec\leq -K,$$ 
		for some $K>0$. Let $E$ be a holomorphic vector bundle on $X$, $D_{E}$ be the  Chern connection on $E$. If the curvature $\Theta(E)$ of  $D_{E}$ such that 
		$$C:=|[\La,i\Theta(E)]|\leq c(n)K,$$
		where $c_{n}$ is a positive constant depends only $n$, then for every $p=0,1,\cdots,n$, the Euler characteristic 
		$$\chi^{p}(X,E):=\int_{X}td(X)ch(\Om^{p,0}(TX)\otimes E)$$
		does not vanish and
		$$sign\chi^{p}(X,E)=(-1)^{n-p}.$$
		Furthermore, for all $0\leq j\leq n$, $(X,E)$ satisfy Chern number inequalities
		\begin{equation*}
		(-1)^{n+j}K_{j}(X,E)\geq\sum_{p=j}^{n}\tbinom{p}{j}.
		\end{equation*}
	\end{theorem} 
	\begin{remark}
		The Chern number inequalities are always not sharp. For example, suppose that  the curvature $\Theta(E)$ of the Chern connection $D_{E}$ is small enough in the sense of $L^{\infty}$-norm. Then there exists a flat connection on $\Gamma$ on $E$ (see \cite{Uhl}). Hence following Proposition \ref{P1}, we have
		$$\chi_{y}(X,E)=rank(E)\chi_{y}(X),$$
		i.e., for every $p=0,1,\cdots,n$, the Euler characteristic satisfies
		$$\chi^{p}(X,E)=rank(E)\chi^{p}(X).$$
		Therefore, $X,E$ satisfy Chern number inequalities
		\begin{equation}\label{E2}
		(-1)^{n+j}K_{j}(X,E)=rank(E)\sum_{p=j}^{n}\tbinom{p}{j}h^{p,n-p}_{(2)}(X,E)\geq rank(E)(-1)^{j}K_{j}(\mathbb{C}P^{n}).
		\end{equation}
		All the equality cases in (\ref{E2}) hold if and only if $\chi^{p}(X)=(-1)^{n-p}$, $0\leq p\leq n$. (see \cite[Theorem 2.1]{Li}). 
	\end{remark}
	In \cite{Gro}, Gromov shown that for every $p=0,1,\cdots,n$, the Euler characteristic of a compact K\"{a}hler hyperbolic manifold satisfies
	$$sign\chi^{p}(X)=(-1)^{n-p},$$
	as a consequence $(-1)^{n}\chi(X)\geq n+1$. Let $L$ be a holomorphic line bundle on a compact K\"{a}hler manifold $X$. We call 
	$$P^{(p)}_{n}(m,L):=\chi^{p}(X,L^{\otimes m})$$
	the $p$-Hilbert polynomial of line bundle $L$. The second propose of this article is to show that the lower bound of the Euler characteristic $(-1)^{n}\chi(X)$ estimated by $K$. 
	\begin{theorem}\label{T4}
		Let $(X,\w)$ be a  compact K\"{a}hler manifold with sectional curvature bounded from above by a negative constant, i.e.,
		$$sec\leq -K,$$ 
		for some $K>0$. Let $L$ be a holomorphic line bundle on $X$. Suppose that the $p$-Hilbert polynomial $\chi^{p}(X,L^{\otimes m})$ is not constant for some $p\in[0,n]$. Then there exists a integer $\tilde{m}=\tilde{m}(p)\in[-\frac{c_{n}K}{nC},\frac{c_{n}K}{nC}]$ such that either
		\begin{equation*}
		(-1)^{n-p}\chi^{p}(X)\geq \lfloor\frac{c_{n}K}{nC} \rfloor+1.
		\end{equation*}
		or
		\begin{equation*}
		(-1)^{n-p}\chi^{p}(X,L^{\otimes\tilde{m}})\geq\lfloor\frac{c_{n}K}{nC} \rfloor+1,
		\end{equation*}
		where $C:=|[\La,i\Theta(L)]|$. In particular, the Euler  characteristic of $X$ satisfies
		$$(-1)^{n}\chi(X)\geq (n+1)+\lfloor\frac{c_{n}K}{nC} \rfloor.$$
	\end{theorem}
	\begin{remark}
		The conclusion of the theorem  valid for all $p=0,\cdots,n$ if the line bundle $L$ satisfies $\int_{X}c^{n}_{1}(L)\neq 0$. Since the canonical bundle $K_{X}$ on a compact K\"{a}hler hyperbolic manifold is ample (see \cite[Theorem 2.11]{CY}), then $X$ is projective, i.e, there is an embedding $i:X\hookrightarrow\mathbb{P}^{N}$. We denote by $\mathcal{O}(1)$ the  tautological line bundle on $\mathbb{P}^{N}$. The pull back bundle $i^{\ast}\mathcal{O}(1)$ of the line bundle $\mathcal{O}(1)$ on $X$ satisfies $\int_{X}c^{n}_{1}(i^{\ast}\mathcal{O}(1))\neq0$.
	\end{remark}
	The lifted K\"{a}hler form $\tilde{\w}$ on the universal covering space $\pi:(\tilde{X},\tilde{\w})\rightarrow(X,\w)$ is $d$(bounded). Set
	$$Q(\w):=\{\theta\in\Om^{1}(\tilde{X}):\tilde{\w}=d\theta\}.$$
	Let $E(\theta):=\inf_{\theta\in Q(\w)}\|\theta\|_{L^{\infty}(\tilde{X})}$. The eigenvalues of the Laplace-Beltrami operator $\De_{\bar{\pa}_{\tilde{E}}}$ on $\Om^{p,q}_{(2)}(\tilde{X},\tilde{E})$ ($p+q\neq n$) have a lower bounded  $c_{n}E(\theta)^{-2}-C$. Then the Euler number of $X$ satisfies
	$$(-1)^{n}\chi(X)\geq (n+1)+\lfloor\frac{c_{n}E(\theta)^{-2}}{nC} \rfloor\geq n+\frac{c_{n}E(\theta)^{-2}}{nC}.$$
	Hence, we get the following result.
	\begin{corollary}
		Let $(X,\w)$ be a  compact K\"{a}hler manifold with sectional curvature bounded from above by a negative constant, i.e.,
		$$sec\leq -K,$$ 
		for some $K>0$. Suppose that there is a   holomorphic line bundle $L$ on $X$ such that the $p$-Hilbert polynomial $\chi^{p}(X,L^{\otimes m})$ is not constant for some $p\in[0,n]$. Then  
		$$\sqrt{n}K^{-\frac{1}{2}}\geq E(\theta)\geq[\frac{c_{n}}{nC((-1)^{n}\chi(X)-n)}]^{\frac{1}{2}}.$$
	\end{corollary}
	In the K\"{a}hler surfaces case, we can get a stronger result as follows.
	\begin{theorem}\label{T2}
		Let $(X,\w)$ be a  compact K\"{a}hler surface with sectional curvature bounded from above by a negative constant, i.e.,
		$$sec\leq -K,$$ 
		for some $K>0$. Suppose that there is a  holomorphic line bundle $L$ on $X$ such that $\int_{X}c^{2}_{1}(L)\neq 0$. Then the Euler characteristic of $X$ satisfies
		\begin{equation*}
		\chi(X)\geq 3+|\int_{X}c^{2}_{1}(L)|\cdot(\lfloor\frac{c_{n}K}{C}\rfloor)^{2}.
		\end{equation*}
		where $C:=|[\La,i\Theta(L)]|$.
	\end{theorem} 
	We denote by $$Z_{p}:=\{m\in\mathbb{R}: P^{(p)}_{n}(m,L)=\chi^{p}(X)\}$$
	the set of real roots of polynomial $P^{(p)}_{n}(m,L)-\chi^{p}(X)$. We denote  
	$$m_{p}(L)=\max_{m\in Z_{p}}|m|.$$	
	\begin{remark}
		Following Corollary \ref{C3},  if then Chern connection of the holomorphic line bundle on  compact K\"{a}hler surface satisfies 
		$$C:=|[\La,i\Theta(L)]|\leq c_{n}k,$$ 
		then $$\int_{X}c_{1}(X)c_{1}(L)=0.$$ 
		For any $p=0,1,2$, we then have (see the proof of Theorem \ref{T2})
		$$m_{p}(L)=0.$$
	\end{remark}
	On higher dimensions case, we have following results.
	\begin{theorem}\label{T5}
		Let $(X,\w)$ be a  compact K\"{a}hler manifold with sectional curvature bounded from above by a negative constant, i.e.,
		$$sec\leq -K,$$ 
		for some $K>0$. Suppose that there is a holomorphic line bundle $L$ on $X$ such that $a_{n}:=\int_{X}c^{n}_{1}(L)\neq0$. Then the Euler characteristic of $X$ satisfies
		$$(-1)^{n}\chi(X)\geq\max\{n+1,n+1+2|a_{n}|sign(\lfloor{c_{n}K-Cm_{p}(L)}\rfloor)(|\lfloor\frac{c_{n}K-Cm_{p}(L)}{2Cn}\rfloor|)^{n} \},$$
		where $C:=|[\La,i\Theta(L)]|$. Furthermore, if $n$ is odd, for any $p=0,1,\dots,n$, we then have
		\begin{equation*}
		(-1)^{n-p}\chi^{p}(X)\geq \max\{1,1+2|a_{n}|sign(\lfloor{c_{n}K-Cm_{p}(L)}\rfloor)(|\lfloor\frac{c_{n}K-Cm_{p}(L)}{2Cn}\rfloor|)^{n}\}.
		\end{equation*}
	\end{theorem}
	
	\section{Preliminaries}
	Let $X$ be a smooth K\"{a}hler manifold with K\"{a}hler form $\w$ and $E$ be a smooth vector bundle over $X$.  We denote by  $\Om^{k}(X,E)$ the space of $C^{\infty}$  sections of the tensor product vector bundle  $\Om^{k}(X)\otimes E$ obtained from $\Om^{k}(X)$ and $E$, i.e., $\Om^{k}(X,E):=\Gamma(\Om^{k}(X)\otimes E)$. We denote by $\Om^{p,q}(X,E)$ the space of $C^{\infty}$ sections of the bundle $\Om^{p,q}(X)\otimes E$. We have a direct sum decomposition
	$$\Om^{k}(X,E)=\bigoplus_{p+q=k}\Om^{p,q}(X,E).$$
	For any connection $d_{A}$ on $E$, we have the covariant exterior derivatives 
	$$d_{A}:\Om^{k}(X)\otimes E\rightarrow\Om^{k+1}(X)\otimes E.$$
	Like the canonical splitting the exterior derivatives $d=\pa+\bar{\pa}$, $d_{A}$ decomposes over $X$ into $$d_{A}=\pa_{A}+\bar{\pa}_{A}.$$
	We will need some of the basic apparatus of Hermitian exterior algebra. Denote by $L$ the operator of exterior multiplication by the K\"{a}hler form $\w$:
	$$L\a=\w\wedge\a,\ \a\in\Om^{p,q}(X,E),$$
	and, as usual, let $\La$ denote its pointwise adjoint, i.e.,
	$$\langle\La\a,\b\rangle=\langle\a,L\b\rangle.$$
	Then it is well known that $\La=\ast^{-1}\circ L\circ \ast$ \cite{Huy}. A basic fact is
	\begin{lemma}\label{L1}
		The map $L:\Om^{p,q}(X,E)\rightarrow\Om^{p+1,q+1}(X,E)$ is injective for $p+q\leq n$.
	\end{lemma}
	The proof is then purely algebraic and can be found in standard texts on geometry. An elegant approach is through representations of $sl_{2}$, see \cite[ Chap.5, Theorem 3.12]{Wel} or \cite{Dem,Huy}.
	
	We recall some definitions on Hermitian vector bundle \cite[Charp V, Section 7]{Dem}. Let $E$ be a Hermitian vector bundle of rank $r$ over a smooth Riemannian manifold $X$, $\dim_{\mathbb{R}}X=n$. We denote respectively by $(\xi_{1},\ldots,\xi_{n})$ and $(e_{1},\ldots,e_{r})$ orthonormal frames on $TX$ and $E$ over an open subset $U\subset X$. The associated inner product of $E$ given by a positive definite Hermitian metric $h_{\la\mu}$ with smooth coefficients on $U$, such that
	$$\langle e_{\la}(x),e_{\mu}(x)\rangle=h_{\la\mu}(x),\ \forall x\in\Om.$$
	When $E$ is Hermitian, one can define a natural sesquilinear map
	$$\Om^{p}(X,E)\times\Om^{q}(X,E)\rightarrow\Om^{p+q}(X,\mathbb{C})$$
	$$(\a,\b)\mapsto tr(s\wedge t)$$
	combining the wedge product of forms with the Hermitian metric on $E$. If $\a=\sum\sigma_{\la}\otimes e_{\la}$, $\b=\sum\tau_{\mu}\otimes e_{\mu}$, we let
	$$tr(\a\wedge\b):=\sum_{1\leq\la,\mu\leq r}\sigma_{\la}\wedge\bar{\tau}_{\mu}\langle e_{\la},e_{\mu}\rangle.$$
	A connection $d_{A}$ said to be compatible with the Hermitian structure of $E$, or briefly Hermitian, if for every $\a\in\Om^{p}(X,E)$, $\b\in\Om^{q}(X,E)$ we have
	\begin{equation}\label{E05}
	dtr(\a\wedge\b)=tr(d_{A}\a\wedge\b)+(-1)^{p}tr(\a\wedge d_{A}\b).
	\end{equation}
	The inner product $\langle\cdot,\cdot\rangle$ on $\Om^{\ast}(X,E)$ defined as, See \cite[Charp VI, Section 3.1]{Dem}
	$$\langle\a,\b\rangle=\ast tr(\a\wedge\ast\b),\ \ \a,\b\in\Om^{p}(X,E).$$
	We denote by $Tr$ the sesquilinear map  $Tr:\Om^{p}(X,EndE)\times\Om^{q}(X,End E)\rightarrow\Om^{p+q}(X,\mathbb{C})$ induced by the map $tr:\Om^{p}(X,E)\times\Om^{q}(X,E)\rightarrow\Om^{p+q}(X,\mathbb{C})$.
	
	There are several commutation relations between the basic operators associated to a K\"{a}hler manifold $X$, all following more or less directly from the K\"{a}hler condition $d\w=0$; taken together, these are referred to as the K\"{a}hler identities \cite{Dem,Huy}. 
	\begin{proposition}\label{P3}
		Let $X$ be a complete K\"{a}hler manifold, $E$ a Hermitian vector bundle over $X$ and $d_{A}$ be a Hermitian connection on $E$. We have the following identities\\
		(i)\ $[\La,\bar{\pa}_{A}]=-\sqrt{-1}\pa^{\ast}_{A}$,\ $[\La,\pa_{A}]=\sqrt{-1}\bar{\pa}^{\ast}_{A}$.\\
		(ii)\ $[\bar{\pa}^{\ast}_{A},L]=\sqrt{-1}\pa_{A}$,\ $[\pa^{\ast}_{A},L]=-\sqrt{-1}\bar{\pa}_{A}$.
	\end{proposition}
	Since $\w$ is parallel, the operator $L^{k}:\Om^{p}(X,E)\rightarrow\Om^{p+2k}(X,E)$ defined by $L^{k}(\a)=\a\wedge\w^{k}$ for all $p$-forms commutes with $d_{A}$. But the operator $L^{k}$ does not commute with $d^{\ast}_{A}$ in general, therefore the operator $L^{k}$ does not commute with $\De_{A}$. 
	
	If $A$ and $B$ are operators on forms, define the (graded) commutator as 
	$$[A,B]=AB-(-1)^{degA\cdot degB}BA,$$
	where $degT$ is the integer $d$ for $T$: $\oplus_{p+q=r}\Om^{p,q}(X,E)\rightarrow\oplus_{p+q=r+d}\Om^{p,q}(X,E)$. If $C$ is another endomorphism of degree $c$, the following $Jacobi$ $identity$ is easy to check
	$$(-1)^{ca}\big{[}A,[B,C]\big{]}+(-1)^{ab}\big{[}B,[C,A]\big{]}+(-1)^{bc}\big{[}C,[A,B]\big{]}=0.$$
	At first, we observe that the operator $L^{k}$ commutes with $\De_{A}$ for any connection $d_{A}\in\mathcal{A}_{E}^{1,1}$.
	\begin{lemma}\label{L3}(\cite[Lemma 3.9]{Hua})
		$$[\De_{A},L^{k}]=2k\sqrt{-1}(F^{2,0}_{A}-F^{0,2}_{A})L^{k-1},\ \forall\ k\in\mathbb{N}.$$
		In particular, if the connection $d_{A}\in\mathcal{A}_{E}^{1,1}$, then $\De_{A}$ commutes with $L^{k}$ for any $k\in\mathbb{N}$.
	\end{lemma}
	\begin{proof}
		The case of $k=1$: the operators $d_{A}$,\ $d^{\ast}_{A}$ and $L$ satisfy the following Jacobi identity:
		$$-\big{[}L,[d_{A},d^{\ast}_{A}]\big{]}+\big{[}d^{\ast}_{A},[L,d_{A}]\big{]}+\big{[}d_{A},[d^{\ast}_{A},L]\big{]}=0.$$
		Then we have
		\begin{equation}\nonumber
		\begin{split}
		[L,\De_{A}]&=\big{[}d_{A},[d^{\ast}_{A},L]\big{]}\\
		&=[\pa_{A}+\bar{\pa}_{A},\sqrt{-1}(\pa_{A}-\bar{\pa}_{A})]\\
		&=[\sqrt{-1}\pa_{A},\pa_{A}]-[\sqrt{-1}\bar{\pa}_{A},\bar{\pa}_{A}]\\
		&=2\sqrt{-1}(F^{2,0}_{A}-F^{0,2}_{A}).\\
		\end{split}
		\end{equation}
		We suppose that the case of $p=k-1$ is true, i.e., $$[\De_{A},L^{k-1}]=2(k-1)\sqrt{-1}(F^{2,0}_{A}-F^{0,2}_{A})L^{k-2}.$$ Thus if $p=k$, we have
		\begin{equation}\nonumber
		\begin{split}
		[\De_{A},L^{k}]&=[\De_{A},L]L^{k-1}+L[\De_{A},L^{k-1}]\\
		&=2\sqrt{-1}(F^{0,2}_{A}-F^{2,0}_{A})L^{k-1}+2(k-1)\sqrt{-1}L(F^{2,0}_{A}-F^{0,2}_{A})L^{k-2}\\
		&=2k\sqrt{-1}(F^{2,0}_{A}-F^{0,2}_{A})L^{k-1}.
		\end{split}
		\end{equation}
		If $d_{A}\in\mathcal{A}_{E}^{1,1}$, then $[\De_{A},L^{k}]=0$.
	\end{proof}
	\section{Harmonic forms on vector bundle $E$}
	As we derive estimates in this section (and also following sections), there will be many constants which appear. To simplify notation we shall write $a\lesssim b$, for $a\leq const_{n}b$, and $a\approx b$, for $b\lesssim a\lesssim b$. 
	\subsection{Uniform positive lower bounds for the least eigenvalue of $\De_{\bar{\pa}_{E}}$}
	Let $(X,g)$ be an oriented, smooth, Riemannian manifold, $\dim_{\mathbb{R}}X=n$, and $E$ be a Hermitian vector bundle over $X$. Assume now that $d_{A}$ is a Hermitian connection on $E$. The formal adjoint operator of $d_{A}$ acting on $\Om^{p}(X,E)$ is $d^{\ast}_{A}=(-1)^{np+1}\ast d_{A}\ast$, where the operator $\ast:\Om^{p}(X,E)\rightarrow\Om^{n-p}(X,E)$ induced by the Hodge-Poincar\'{e}-de Rahm operator $\ast_{g}$ \cite[Charp Vi, Section 3]{Dem}. Indeed, if $\a\in\Om^{p}(X,E)$, $\b\in\Om^{p+1}(X,E)$ have compact support, we get
	$$\int_{X}\langle d_{A}\a,\b\rangle=\int_{X}\langle \a,d^{\ast}_{A}\b\rangle.$$
	The Laplace-Beltrami operator associated to $d_{A}$ is the second order operator $\De_{A}=d_{A}d^{\ast}_{A}+d^{\ast}_{A}d_{A}$. The space of $L^{2}$-harmonic forms of degree of $k$ respect to the Laplace-Beltrami operator $\De_{A}$ is defined by
	$$\mathcal{H}^{k}_{(2)}(X,E)=\{\a\in\Om^{k}_{(2)}(X,E):\De_{A}\a=0\}.$$
	Define the $\de$-Laplacian by setting $\De_{\de}:=[\de,\de^{\ast}]$. For all $(p,q)$, we denote by 
	$$\mathcal{H}^{p,q}_{(2);\de}(X,E):=\ker(\De_{\de})\cap\Om^{p,q}_{(2)}(X,E)$$
	the space of $L^{2}$ $\de$-harmonic forms in bidegree $(p,q)$. We have an useful lemma as follows.
	\begin{lemma}(\cite[Lemma 3.2]{Hua})
		Let $X$ be a complete Riemannian manifold $X$, $E$ a Hermitian vector bundle over $X$. Then 
		$$\mathcal{H}^{k}_{(2)}(X,E)=\ker d_{A}\cap\ker d_{A}^{\ast}\cap\Om^{k}_{(2)}(X,E),$$
		$$\mathcal{H}^{p,q}_{(2);\de}(X,E)=\ker \de\cap\ker \de^{\ast}\cap\Om^{p,q}_{(2)}(X,E),$$
		where $\de=\bar{\pa}_{A}$ or $\pa_{A}$.
	\end{lemma}
	\begin{theorem}\label{T10}
		Let $(X,\w)$ be a complete, K\"{a}hler manifold, $\dim_{\C}X=n$, with a $d$(bounded) K\"{a}hler form $\w$, i.e., there is a bounded $1$-form $\theta$ such that $\w=d\theta$, $d_{A}\in\mathcal{A}_{E}^{1,1}$ be a smooth Hermitian integrable connection on a Hermitian vector bundle $E$ over $X$. Then 
		$$\mathcal{H}^{k}_{(2)}(X,E)=\{0\},\ \forall\ k\neq n.$$	
	\end{theorem}
	\begin{proof}
		Let $k<n$. For every $d_{A}$-closed $L^{2}$ $k$-form $\a$, the form 
		$$L^{n-k}\a=\w^{n-k}\wedge\a=d_{A}(\w^{n-k-1}\wedge\theta\wedge\a)$$ 
		is $L^{2}$. We denote $\b=\w^{n-k-1}\wedge\theta\wedge\a$. One can see that $\b$ is $L^{2}$. In particular, if $\a$ is $\De_{A}$-harmonic, then $L^{k}\a=0$ is also   $\De_{A}$-harmonic. Following  Proposition \cite[Proposition 3.7]{Hua}, we get $L^{n-k}\a=0$. This implies, by Lemma \ref{L1}, that $\a=0$. The case $k>n$ follows by $E^{\ast}$ is a holomorphic verctor bundle on $X$ and $\mathcal{H}^{k}_{(2)}(X,E)\cong\mathcal{H}^{2n-k}_{(2)}(X,E^{\ast})=\{0\}$.
	\end{proof}
	We want to sharpen the Lefschetz vanishing theorem \ref{T10} by giving a lower bound on the spectrum of the Laplace operator $\De_{A}$ on $L^{2}$-forms $\Om^{k}(X,E)$ for $k\neq n$.
	\begin{theorem}\label{T3}(\cite[Theorem 1.3]{Hua})
		Let $(X,\w)$ be a complete, K\"{a}hler manifold, $\dim_{\C}X=n$, with a $d$(bounded) K\"{a}hler form $\w$, i.e., there is a bounded $1$-form $\theta$ such that $\w=d\theta$, $d_{A}\in\mathcal{A}_{E}^{1,1}$ be a smooth Hermitian integrable connection on a Hermitian vector bundle $E$ over $X$. If $\a\in\Om^{k}_{(2)}(X,E)$ such that $\De_{A}\a\in L^{2}$, ($k\neq n)$, then we have the inequality
		\begin{equation}\label{E08}
		c_{n,k}\|\theta\|^{-2}_{L^{\infty}(X)}\|\a\|^{2}_{L^{2}(X)}\leq\langle\a,\De_{A}\a\rangle_{L^{2}(X)},
		\end{equation}
		where $c_{n,k}>0$ is a constant which depends only on $n,k$. 
	\end{theorem}
	\begin{proof}
		Let $\a$ be a $p$-form on vector bundle ($p<n$), we denote $\b=L^{k}\a=\w^{k}\wedge\a$. We recall the operator $L^{k}:\Om^{p}(X,E)\rightarrow\Om^{p+2k}(X,E)$ for a given $p<n$ and $p+k=n$. Since the Lefschetz theorem $L^{k}$ is a bijective quasi-isometry,  
		$$\|\a\|_{L^{2}(X)}\thickapprox\|\b\|_{L^{2}(X)}.$$ 
		If $\a$ is in $L^{2}$, $\b$ is also in $L^{2}$. Since $F_{A}^{0,2}=0$, following Lemma \ref{L3}, $[L^{k},\De_{A}]=0$. Then we have
		$$\langle\De_{A}\b,\b\rangle_{L^{2}(X)}=\langle L^{k}(\De_{A}\a)\,L^{k}\a\rangle_{L^{2}(X)}\approx\langle\De_{A}\a,\a\rangle_{L^{2}(X)}.$$
		We write $\b=d_{A}\eta-\tilde{\a}$, for $\eta=\a\wedge\w^{k-1}\wedge\theta$ and $\tilde{\a}=d_{A}\a\wedge\w^{k-1}\wedge\theta$. Observe that
		\begin{equation}\label{E01}
		\|\eta\|_{L^{2}(X)}\lesssim\|\theta\|_{L^{\infty}(X)}\|\a\|_{L^{2}(X)}\lesssim\|\theta\|_{L^{\infty}(X)}\|\b\|_{L^{2}(X)},
		\end{equation}
		and
		\begin{equation}\label{E02}
		\begin{split}
		\|\tilde{\a}\|_{L^{2}(X)}&\lesssim\|d_{A}\a\|_{L^{2}(X)}\|\theta\|_{L^{\infty}(X)}\lesssim\langle\De_{A}\a,\a\rangle_{L^{2}(X)}^{1/2}\|\theta\|_{L^{\infty}(X)}.\\
		\end{split}
		\end{equation}
		We then have
		\begin{equation}\nonumber
		\begin{split}
		\|\b\|^{2}_{L^{2}(X)}&\leq|\langle\b,d_{A}\eta\rangle_{L^{2}(X)}|+|\langle\b,\tilde{\a}\rangle_{L^{2}(X)}|\\
		&\leq|\langle d_{A}^{\ast}\b,\eta\rangle_{L^{2}(X)}|+|\langle\b,\tilde{\a}\rangle_{L^{2}(X)}|\\
		&\lesssim\langle\De_{A}\b,\b\rangle_{L^{2}(X)}^{1/2}\|\theta\|_{L^{\infty}(X)}\|\b\|_{L^{2}(X)}
		+\|\b\|_{L^{2}(X)}\|d_{A}\a\|_{L^{2}(X)}\|\theta\|_{L^{\infty}(X)}\\
		&\lesssim\langle\De_{A}\a,\a\rangle_{L^{2}(X)}^{1/2}\|\theta\|_{L^{\infty}(X)}\|\b\|_{L^{2}(X)}.\\
		\end{split}
		\end{equation}
		This yields the desired estimation
		$$
		\|\a\|^{2}_{L^{2}(X)}\lesssim\|\b\|^{2}_{L^{2}(X)}
		\lesssim\|\theta\|^{2}_{L^{\infty}(X)}\langle\De_{A}\a,\a\rangle_{L^{2}(X)}.$$
		The case $p>n$ follows by $E^{\ast}$ is a holomorphic verctor bundle on $X$, the Poincar\'{e} duality as the operator $\ast_{E}:\Om^{p}(X,E)\rightarrow\Om^{2n-p}(X,E^{\ast})$ commutes with $\De_{A}$ and is isometric for the $L^{2}$-norms.
		
	\end{proof}
	\begin{lemma}
		If $d_{A}\in\mathcal{A}^{1,1}_{E}$, then
		$$\De_{A}=\De_{\bar{\pa}_{E} }+\De_{\pa_{E}}.$$
	\end{lemma}
	\begin{proof}
		Following the definitions of $\De_{A}$, $\De_{\bar{\pa}_{E}}$ and $\De_{\pa_{E}}$, we have
		\begin{equation*}
		\begin{split}
		\De_{A}&=[\pa_{E}+\bar{\pa}_{E},\pa^{\ast}_{E}+\bar{\pa}^{\ast}_{E}]\\
		&=[\pa_{E},\pa^{\ast}_{E}]+[\bar{\pa}_{E},\bar{\pa}^{\ast}_{E}]+[\pa_{E},\bar{\pa}^{\ast}_{E}]+[\bar{\pa}_{E},\pa^{\ast}_{E}]\\
		&=\De_{\bar{\pa}_{E}}+\De_{\pa_{E}}+[\pa_{E},\bar{\pa}^{\ast}_{E}]+[\bar{\pa}_{E},\pa^{\ast}_{E}].\\
		\end{split}
		\end{equation*}	
		Following identities on Proposition \ref{P3}, we have
		\begin{equation*}
		\begin{split}
		[\pa_{E},\bar{\pa}^{\ast}_{E}]&=[\pa_{E},-i[\La,\pa_{E}]]\\
		&=-i[\La,[\pa_{E},\pa_{E}]]-i[\pa_{E},[\pa_{E},\La]]\\
		&=i[\pa_{E},[\La,\pa_{E}]].\\
		\end{split}
		\end{equation*}
		Therefore,
		$$[\pa_{E},\bar{\pa}^{\ast}_{E}]=0.$$
		By the similar way, we also get 
		$$[\bar{\pa}_{E},\pa^{\ast}_{E}]=0.$$
		Therefore, we have
		$$\De_{A}=\De_{\bar{\pa}_{E} }+\De_{\pa_{E}}.$$
	\end{proof}
	\begin{proposition}\label{P2}
		Let $(X,\w)$ be a complete, K\"{a}hler manifold, $\dim_{\C}X=n$, with a $d$(bounded) K\"{a}hler form $\w$, i.e., there is a bounded $1$-form $\theta$ such that $\w=d\theta$, $D_{E}$ be the Chern connection on a holomorphic Hermitian vector bundle $E$ over $X$. Then for any $\a\in\Om^{p,q}_{(2)}(X,E)$, $(k:=p+q\neq n)$, such that $\De_{\bar{\pa}_{E} }\a\in L^{2}$, which satisfies the inequality
		\begin{equation*}
		(c(n,k)\|\theta\|^{-2}_{L^{\infty}(X)}-|[ i\Theta(E),\La]|)\|\a\|_{L^{2}(X)}\leq\langle\De_{\bar{\pa}_{E}}\a,\a\rangle_{L^{2}(X)}.
		\end{equation*}
		where $c_{n,k}>0$ is a constant which depends only on $n,k$. Furthermore. if  $$|[i\Theta(E),\La]|\leq c_{n}\|\theta\|^{-2}_{L^{\infty}(X)},$$ 
		where $c_{n}$ is a uniformly positive constant only depends on $n$ which satisfies $c_{n}<\inf c_{n,k}$, we then have
		$$\mathcal{H}^{p,q}_{(2);\bar{\pa}_{E}}(X,E)=0,\ \forall\ p+q\neq n.$$
	\end{proposition}
	\begin{proof}
		Following the Bochner-Kodaira-Nakano formula \cite[Chapter VII., Corollary 1.3]{Dem} $$\De_{\bar{\pa}_{E}}=\De_{\pa_{E}}+[i\Theta(E),\La],$$
		we have
		$$\De_{E}=\De_{\bar{\pa}_{E}}+\De_{\pa_{E}}=2\De_{\bar{\pa}_{E}}-[i\Theta(E),\La],$$
		where $\De_{E}:=D_{E}D_{E}^{\ast}+D_{E}^{\ast}D_{E}$. Then for any $\a\in\Om^{p,q}_{(2)}(X,E)$, $(p+q\neq n)$, we have
		\begin{equation*}
		\begin{split}
		\langle\De_{E}\a,\a\rangle_{L^{2}(X)}&\leq 2\langle\De_{\bar{\pa}_{E}}\a,\a\rangle_{L^{2}(X)}+|\langle[i\Theta(E),\La]\a,\a\rangle_{L^{2}(X)}|\\
		&\leq2\langle\De_{\bar{\pa}_{E}}\a,\a\rangle_{L^{2}(X)}+|[ i\Theta(E),\La]|\cdot\|\a\|_{L^{2}(X)}|.\\
		\end{split}
		\end{equation*}
		We then have
		$$\langle\De_{\bar{\pa}_{E}}\a,\a\rangle_{L^{2}(X)}\geq (c(n,k)\|\theta\|^{-2}_{L^{\infty}(X)}-|[ i\Theta(E),\La]|)\|\a\|_{L^{2}(X)}.$$
		where $c(n,k)$ is a uniformly positive constant. 
		
		For any $k\neq n$, if 
		$$|[\La,i\Theta(E)]|\leq c_{n}\|\theta\|^{-2}_{L^{\infty}(X)}<c_{n,k}\|\theta\|^{-2}_{L^{2}(X)},$$
		then every $\a\in\mathcal{H}_{(2);\bar{\pa}_{E} }^{p,q}(X,E)$, we get
		$$0\leq (c(n,k)\|\theta\|^{-2}_{L^{\infty}(X)}-|[ i\Theta(E),\La]|)\|\a\|_{L^{2}(X)}\leq 0,$$
		i.e., $\a=0$.  We complete the proof of this theorem. 
	\end{proof}
	A compact K\"{a}hler manifold $(X,J,\w)$ with sectional curvature bounded form above by a negative constant, i.e., $sec\leq -K$ for some $K>0$. We denote by $(\tilde{X},\tilde{J},\tilde{\w})$ the universal covering space of $(X,J,\w)$. Since $\pi$ is local isometry, the sectional curvature of $\tilde{X}$ also bounded form above by the negative constant $K$. By \cite[Lemma 3.2]{CY}, there exists $1$-form $\theta$ on $\tilde{X}$ such that  
	$$\tilde{\w}=d\theta$$
	and
	$$ \|\theta\|_{L^{\infty}(\tilde{X})}\leq \sqrt{n}K^{-\frac{1}{2}}.$$
	\begin{corollary}\label{C2}
		Let $(X,\w)$ be a  compact K\"{a}hler manifold with sectional curvature bounded from above by a negative constant, i.e.,
		$$sec\leq -K,$$ 
		for some $K>0$.  Let $E$ be a holomorphic vector bundle on $X$, $D_{E}$ be the  Chern connection on $E$. Let $\pi:(\tilde{X},\tilde{g}\rightarrow(X,g)$ be the universal covering with $\tilde{g}=\pi^{\ast}g$, $\tilde{E}=\pi^{\ast}E$ the pull back bundle over $\tilde{X}$. Then for any $\a\in\Om^{p,q}_{(2)}(\tilde{X},\tilde{E})$ such that $\De_{\bar{\pa}_{\tilde{E}} }\a\in L^{2}$, which satisfies the inequality
		\begin{equation*}
		(c(n,k)K/n-|[ \La, i\Theta(E)]|)\|\a\|_{L^{2}(X)}\leq\langle\De_{\bar{\pa}_{\tilde{E}}}\a,\a\rangle_{L^{2}(\tilde{X})}.
		\end{equation*}
		where $c_{n,k}>0$ is a constant which depends only on $n,k$. Furthermore. if  $$|[\La, i\Theta(E)]|\leq c_{n}K,$$
		where $c_{n}$ is a  positive constant only depends on $n$ which satisfies $c_{n}<\inf c_{n,k}/n$, we then have
		$$\mathcal{H}^{p,q}_{(2);\bar{\pa}_{\tilde{E}}}(\tilde{X},\tilde{E})=0,\ \forall\ p+q\neq n.$$
	\end{corollary}	
	\subsection{Nonvanishing results}
	In \cite{Gro}, Gromov proved a nonvanishing for $p+q=dim_{\C}X$ follows from the $L^{2}$-index theorem and an upper bound for the bottom of the spectrum \cite[Main Theorem]{Gro}.  A special case of a conjecture of Hopf follows from the main theorem. Namely, the Euler characteristic $\chi(X)$ of a compact, negatively curved K\"{a}hler manifold $X$ of complex dimension $n$ satisfies $sign\chi(X)=(-1)^{n}$. Let $E$ be a holomorphic  vector bundle equipped with a Hermitian metric and Hermitian connection $d_{A}$ over a compact K\"{a}hler manifold. Suppose $X$ is a compact K\"{a}hler manifold with underlying Riemann metric $g$ . We denote by $\na^{g}$ the Hermitian connections induced by the Levi-Civita connection on $\Om^{\bullet,\bullet}TX$. Let $D_{E}$ be the Chern connection on $E$. Thus the twist bundle $\Om^{p,0}TX\otimes E$ is also a holomorphic vector bundle on $X$. We denote by $\chi^{p}(X,E)$ the index of the operator 
	$$\mathcal{D}_{p}=\bar{\pa}_{E}+\bar{\pa}_{E}^{\ast}:\Om^{p,\ast}(X,E)\rightarrow\Om^{p,\ast\pm1}(X,E).$$ 
	By definition 
	\begin{equation*}
	\begin{split}
	\chi^{p}(X,E)&=Index(\mathcal{D}_{p})\\
	&=\dim_{\C}(\ker\mathcal{D}_{p})-\dim_{\C}(\rm{coker}\mathcal{D}_{p})\\
	&=\dim_{\C}\oplus_{q\ even}\mathcal{H}^{p,q}_{\bar{\pa}_{E}}-\dim_{\C}\oplus_{q\ odd}\mathcal{H}^{p,q}_{\bar{\pa}_{E}}\\
	&=\sum_{q=0}^{n}(-1)^{q}h^{p,q}(X,E),\\
	\end{split}
	\end{equation*}
	where   
	$$\mathcal{H}^{p,q}_{\bar{\pa}_{E}}=\{\a\in\Om^{p,q}(X,E):\mathcal{D}_{p}\a=0  \}$$ are the spaces of $\bar{\pa}_{E}$-harmonic forms and $h^{p,q}(X,E):=\dim\mathcal{H}^{p,q}_{\bar{\pa}_{E}}$ the Hodge numbers of $(X,E)$. 
	In particular, $\chi^{0}(X,E)$ called the \textit{Euler-Poincar\'{e}} characteristic \cite[Section 5]{Huy}. The Hirzebruch-Riemann-Roch theorem gives 
	$$\chi^{p}(X,E)=\int_{X}td(X)ch(\Om^{p,0}TX\otimes E)=\int_{X}td(X) ch(\Om^{p,0}TX)ch(E).$$
	
	Given a compact $n$-dimensional manifold $X$, one can associate polynomial $\chi_{y}(X)$, called the Hirzebruch $\chi_{y}$-genus, in terms of their Hodge number 
	$$h^{p,q}(X):=\dim\mathcal{H}^{p,q}_{\bar{\pa}}(X)$$ as follows:
	$$\chi_{y}(X):=\sum_{p=0}^{n}\chi^{p}(X)\cdot y^{p}:=\sum_{p=0}^{n}[\sum_{q=0}^{n}(-1)^{q}h^{p,q}(X)]y^{p}.$$
	where $\chi^{p}(X):=\sum_{q=0}^{n}(-1)^{q}h^{p,q}(X)$ $(0\leq p\leq n)$. The $\chi^{p}(X)$-genus was first introduced by Hirzebruch \cite{Hir}.  On a holomorphic bundle over compact complex manifold, we also define a polynomial as follows:
	$$\chi_{y}(X,E):=\sum_{p=0}^{n}\chi^{p}(X,E)\cdot y^{p}=\sum_{p=0}^{n}[\sum_{q=0}^{n}(-1)^{q}h^{p,q}(X,E)]y^{p}.$$
	The general form of the Hirzebruch-Riemann-Roch theorem, which is a corollary of the Atiyah-Singer index theorem, allows us to compute $\chi_{y}(X,E)$ in terms of the Chern numbers of $X,E$ as follows:
	$$\chi_{y}(X,E)=\int_{X}td(X)ch(\oplus_{p=0}^{n}\Om^{p,0}(TX)y^{p})ch(E).$$
	Let $\gamma_{i}$ denote the formal Chern roots of $TX$ (see \cite[Corollary 5.14]{Huy}), i.e., $i$-th elementary symmetric polynomial of $\gamma_{1},\cdots,\gamma_{n}$ represents the $i$-th Chern class of $(X,J)$:
	$$c_{1}=\gamma_{1}+\cdots+\gamma_{n},\ \ c_{2}=\sum_{1\leq i<j\leq n}\gamma_{i}\gamma_{j},\cdots,\ c_{n}=\gamma_{1}\cdots\gamma_{n}.$$
	Then
	$$td(X)=\prod_{i=1}^{n}\frac{\gamma_{i}}{1-e^{-\gamma_{i}}}$$ and
	$$ch(\oplus_{p=0}^{n}\Om^{p,0}(TX)y^{p})=\prod_{i=1}^{n}(1+ye^{-\gamma_{i}}).$$
	\begin{proposition}\label{P1}
		$$\chi_{y}(X,E)=\int_{X}ch(E)\prod_{i=1}^{n}(1+ye^{-\gamma_{i}})\frac{\gamma_{i}}{1-e^{-\gamma_{i}}}.$$
		In particular, If $E$ is a flat bundle, then
		$$\chi_{y}(X,E)=rank(E)\int_{X}\prod_{i=1}^{n}(1+ye^{-\gamma_{i}})\frac{\gamma_{i}}{1-e^{-\gamma_{i}}}=rank(E)\chi_{y}(X).$$
	\end{proposition}
	\begin{proof}
		If $E$ is a flat bundle, there exists a flat connection $d_{\Gamma}$ on the Hermitian vector bundle $E$. We can write the connection $d_{A}=d_{\Gamma}+a$, where $a$ is a $1$-form take value in $End(E)$. Therefore, $F_{A}=d_{\Gamma}a+a\wedge a$. Then, $[ch(E)]=[Tr(\exp \frac{\textit{i}}{2\pi}F_{A})]=rank(E)$, i.e, there exists a differential form $\eta$ such $$ch(E)=rank(E)+d\eta.$$
		Noting that $d(td(X)ch(\Om^{p,0}(X))=0$. We then have, 
		\begin{equation}\nonumber
		\begin{split}
		\chi_{y}(X,E)&=\int_{X}td(X)ch(\oplus_{p=0}^{n}\Om^{p,0}(TX)y^{p})(rank(E)+d\eta)\\
		&=rank(E)\int_{X}td(X)ch(\oplus_{p=0}^{n}\Om^{p,0}(TX)y^{p})+\int_{X}d(td(X)ch(\oplus_{p=0}^{n}\Om^{p,0}(TX)y^{p})\wedge\eta)\\
		&=rank(E)\chi_{y}(X).
		\end{split}
		\end{equation}
	\end{proof}
	\begin{remark}
		The $\chi_{y}$-genus famously satisfies
		$$\chi_{y}(X)=(-y)^{n}\cdot\chi_{y^{-1}}(X),$$
		which are equivalent to the relations $\chi^{p}(X)=(-1)^{n}\chi^{n-p}(X)$ and can be derived from the Serre duality for the Hodge number:
		\begin{equation*}
		\begin{split}
		\chi^{p}(X)&=\sum_{q=0}^{n}(-1)^{q}h^{p,q}(X)=\sum_{q=0}^{n}(-1)^{q}h^{n-p,n-q}(X)\\
		&=(-1)^{n}\sum_{q=0}^{n}(-1)^{q}h^{n-p,q}(X)=(-1)^{n}\chi^{n-p}(X).\\
		\end{split}
		\end{equation*}
		But for any holomorphic vector bundle $E$ over a compact complex manifold $X$ there exists $\mathbb{C}$-linear isomorphisms (Serre duality \cite[Corollary 4.1.16]{Huy}): 
		$$\mathcal{H}^{p,q}_{\bar{\pa}_{E}}(X,E)\cong\mathcal{H}^{n-p,n-q}_{\bar{\pa}_{E}}(X,E^{\ast}),$$
		so $\chi_{y}(X,E)$ always cannot satisfies $\chi_{y}(X,E)=(-y)^{n}\cdot\chi_{y^{-1}}(X,E)$.  
		
		We also observe that 
		$$\chi_{y}(X,E)|_{y=0}=\chi^{0}(X,E)=\int_{X}td(X)ch(E)$$ 
		and
		$$\chi_{y}(X,E)|_{y=-1}=\int_{X}ch(E)\prod_{i=1}^{n}\gamma_{i}=rank(E)\chi(X).$$ 
	\end{remark}
	Let $X$ be a Riemannian manifold and $\Ga$ a discrete group of isometrics of $X$, such that the differential operator $\mathcal{D}$ commutes with the action of $\Ga$. This presupposes that the action of $\Ga$ lifts to the pertinent bundles $E$ and $E'$, and then the commutation between the actions of $\Ga$ on sections of $E$ and $E'$ and $\mathcal{D}: C^{\infty}(E)\rightarrow C^{\infty}(E')$ makes sense. A trivial example is that of Galois action for a covering map $X\rightarrow X_{0}$, where $\mathcal{D}$ is the pull back from an operator on $X_{0}$. We consider a $\Ga$-invariant Hermitian line bundle $(L,\na)$ on $X$ we assume $X/\Ga$ is compact, and we state Atiyah's $L^{2}$-index theorem for $\mathcal{D}\otimes\na$.
	\begin{theorem}\cite[Theorem 2.3.A]{Gro}\label{T7}
		Let $\mathcal{D}$ be a first-order elliptic operator. Then there exists a closed nonhomogeneous form
		$$I_{\mathcal{D}}=I^{0}+I^{1}+\cdots+I^{n}\in\Om^{\ast}(X)=\Om^{0}\oplus\Om^{1}\oplus\cdots\oplus\Om^{n}$$
		invariant under $\Ga$, such that the $L^{2}$-index of the twisted operator $\mathcal{D}\otimes\na$ satisfies
		\begin{equation}\label{E11}
		L^{2}Index_{\Ga}(\mathcal{D}\otimes\na)=\int_{X/\Ga}I_{\mathcal{D}}\wedge\exp{[\w]},
		\end{equation}
		where $[\w]$ is the Chern form of $\na$, and
		$$\exp{[\w]}=1+[\w]+\frac{[\w]\wedge[\w]}{2!}+\frac{[\w]\wedge[\w]\wedge[\w]}{3!}+\cdots.$$
	\end{theorem}
	\begin{remark}
		(1)  $L^{2}Index_{\Ga}(\mathcal{D}\otimes\na)\neq 0$ implies that
		either $\mathcal{D}\otimes\na$ or its adjoint has a non-trivial $L^{2}$-kernel.\\
		(2) The operator $\mathcal{D}$ used in the present paper is the operator $\bar{\pa}_{E}+\bar{\pa}_{E}^{\ast}$. In this case the $I_{0}$-component of $I_{\mathcal{D}}$ is non-zero. Hence $\int_{X/\Ga}I_{\mathcal{D}}\wedge\exp{\a[\w]}\neq0$, for almost all $\a$, provided the curvature form $[\w]$ is ``homologically nonsingular" $\int_{X/\Ga}[\w]^{n}\neq0$.
	\end{remark}
	We may start with $\Ga$ acting on $(L,\na)$ and then pass (if the topology allows) to the $k$-th root $(L,\na)^{\frac{1}{k}}$ of $(L,\na)$ for some $k>2$. Since the bundle $(L,\na)^{\frac{1}{k}}$ is only defined up to an isomorphism, the action of $\Ga$ does not necessarily lift to $L$. Yet there is a larger group $\Ga_{k}$ acting on $(L,\Ga)$, where $0\rightarrow\mathbb{Z}/k\mathbb{Z}\rightarrow\Ga_{k}\rightarrow\Ga\rightarrow1$. In the general case where $\w(\na)$ is $\Ga$-equivariant, the action of $\Ga$ on $(L,\na)$ is defined up to the automorphism group of $(L,\na)$ which is the circle group $S^{1}=\mathbb{R}/\mathbb{Z}$ as we assume $X$ is connected. Thus we have a non-discrete group, say $\bar{\Ga}$, such that $1\rightarrow S^{1}\rightarrow\bar{\Ga}\rightarrow\Ga\rightarrow1$, and such that the action of $\Ga$ on $X$ lifts to that of $\Ga$ on $(L,\na)$. This gives us the action of $\bar{\Ga}$ on the spaces of sections of $E\otimes L$ and $E'\otimes L$, and we can speak of the $\bar{\Ga}$-dimension of $\ker(\mathcal{D}\otimes\na)$ and $\rm{coker}(\mathcal{D}\otimes\na)$. The proof by Atiyah of the $L^{2}$-index theorem does not change a bit, and the formula (\ref{E11}) remains valid with $\bar{\Ga}$ in place of $\Ga$,
	$$L^{2}Index_{\bar{\Ga}}(\mathcal{D}\otimes\na)=\int_{X/\Ga}I_{\mathcal{D}}\wedge\exp[\w].$$
	Here again, the relevant fact is the implication
	$$\int_{X/\Gamma}I_{\mathcal{D}}\exp[\w]>0\Rightarrow\ker\mathcal{D}\otimes\na\neq0.$$
	Gromov defined the lower spectral bound $\la_{0}=\la_{0}(\mathcal{D})\geq 0$ as the upper bounded of the negative numbers $\la$, such that $\|\mathcal{D}e\|_{L^{2}}\geq\la\|e\|_{L^{2}}$ for those sections $e$ of $E$ where $\mathcal{D}e$ in $L^{2}$. Let $\mathcal{D}$ be a $\Ga$-invariant elliptic operator on $X$ of the first order, and let $I_{\mathcal{D}}=I^{0}+I^{1}+\cdots+I^{n}\in\Om^{\ast}(X)$ be the corresponding index form on $X$. Let $\w$ be a closed $\Ga$-invariant $2$-form on $X$ and denote by $I_{\a}^{n}$ the top component of product $I_{\mathcal{D}}\wedge\exp{\a\w}$, for $\a\in\mathbb{R}$. Hence $I_{\a}^{n}$ is an $\Ga$-invariant $n$-form on $X$, $\dim X=n$ depending on parameter $\a$.
	\begin{theorem}(\cite[2.4.A. Theorem]{Gro})\label{T6}
		Let $H^{1}_{dR}(X)=0$ and let $X/\Ga$ be compact and $\int_{X/\Ga}I_{\a}^{n}\neq 0$, for some $\a\in\mathbb{R}$. If the form $\w$ is $d$(bounded), then either $\la_{0}(\mathcal{D})=0$ or $\la_{0}(\mathcal{D}^{\ast})=0$, where $\mathcal{D}^{\ast}$ is the adjoint operator.
	\end{theorem}
	
	We then have 
	\begin{theorem}\label{T3.12}
		Let $(X,\w)$ be a  compact K\"{a}hler manifold with sectional curvature bounded from above by a negative constant, i.e.,
		$$sec\leq -K,$$ 
		for some $K>0$. Let $E$ be a holomorphic vector bundle on $X$, $D_{E}$ be the  Chern connection on $E$. If the curvature $\Theta(E)$ of $D_{E}$ such that 
		$$|[\La,i\Theta(E)]|\leq c(n)K,$$
		then the spaces of $L^{2}$ $\De_{\bar{\pa}_{\tilde{E}}}$-harmonic $(p,q)$-forms on the lifting bundle $\tilde{E}$ satisfy 
		\begin{equation*}
		\left\{
		\begin{aligned}
		\mathcal{H}^{p,q}_{(2);\bar{\pa}_{\tilde{E}}}(\tilde{X},\tilde{E})=\{0\},\  & p+q\neq n\\
		\mathcal{H}^{p,q}_{(2);\bar{\pa}_{\tilde{E}}}(\tilde{X},\tilde{E})\neq \{0\},\  & p+q=n\\
		\end{aligned}  
		\right.
		\end{equation*}
	\end{theorem}
	\begin{proof}
		Since $\pi$ is a local isometry, the Chern curvature $\Theta(\tilde{E})$ also satisfies
		$$|[\La,i\Theta(\tilde{E})]|\leq c(n)K.$$
		We denote by $\tilde{\mathcal{D}}_{p}$ the lifted of  $\mathcal{D}_{p}$ for $p\geq0$. Following Proposition \ref{P2}, we get $\mathcal{H}^{p,q}_{(2);\bar{\pa}_{\tilde{E}}}(\tilde{X},\tilde{E})=\{0\}$ for any $p+q\neq n$. Since $\int_{X}[\w]^{n}\neq 0$, by Theorem \ref{T6}, we obtain that either 
		$$\ker \tilde{\mathcal{D}}_{p}\cap\oplus\Om^{p,+}_{(2)}(X,E)=\oplus_{q\ even}\mathcal{H}^{p,q}_{(2);\bar{\pa}_{\tilde{E}}}(\tilde{X},\tilde{E})\neq{0}$$ or 
		$$\ker \tilde{\mathcal{D}}_{p}\cap\oplus\Om^{p,-}_{(2)}(X,E)=\oplus_{q\ odd}\mathcal{H}^{p,q}_{(2);\bar{\pa}_{\tilde{E}}}(\tilde{X},\tilde{E})\neq{0}.$$
		Therefore, for any $p+q=n$, we have $$\mathcal{H}^{p,q}_{(2);\bar{\pa}_{\tilde{E}}}(\tilde{X},\tilde{E})\neq{0}.$$
	\end{proof}
	\subsection{$L^{2}$-Hodge numbers on vector bundle $E$}
	We assume throughout this subsection that $(X,g,J)$ is a compact complex $n$-dimensional manifold with a Hermitian metric $g$, and $\pi: (\tilde{X},\tilde{g},\tilde{J})\rightarrow(X, g,J)$ its universal covering with $\Gamma=\pi_{1}(X)$ as an isometric group of deck transformations. Let $E\rightarrow X$ be a holomorphic bundle on $X$. We denote by $\tilde{g}:=\pi^{\ast}g$ the pull-back metric on $\tilde{X}$ and $\tilde{E}:=\pi^{\ast}E$ the pull-back bundle on $\tilde{X}$. We call an open set $U\subset\tilde{X}$ a fundamental domain of the action of $\tilde{\Gamma}$ on $\tilde{X}$ if the following conditions are satisfied:\\
	(1) $\tilde{X}=\cup_{\gamma\in\Gamma}\gamma(\bar{U})$,\\
	(2) $\gamma_{1}(U)\cap \gamma_{2}(U)=\empty$ for $\gamma_{1},\gamma_{2}\in\Gamma$, $\gamma_{1}\neq\gamma_{2}$ and\\
	(3) $\bar{U}\backslash U$ has zero measure.\\
	We then have (see \cite{Bei} or \cite[Section 3.6.1]{MM})
	$$\Om^{p,q}_{(2)}(\tilde{X},\tilde{E})\cong L^{2}(\Gamma)\otimes\Om^{p,q}_{(2)}(U,\tilde{E}|_{U})\cong L^{2}(\Gamma)\otimes\Om^{p,q}(X,E),$$
	where a basis for $L^{2}(\Gamma)$ is given by the function $\de_{\gamma}$ with $\gamma\in\Gamma$ defined by $\de_{\gamma}(\gamma')=1$ if $\gamma=\gamma'$ and $\de_{\gamma}(\gamma')=0$ if $\gamma\neq\gamma'$. Consider now a $\Gamma$-module $V\subset\Om^{p,q}_{(2)}(\tilde{X},\tilde{E})$, that is a closed subspace of $\Om^{p,q}_{(2)}(\tilde{X},\tilde{E})$ which is invariant under the action of $\Gamma$. If $\{\eta_{i}\}_{i\in\N}$ is an orthonormal basis for $V$ then the following quantity is finite:
	$$\sum_{i\in\N}\int_{U}|\eta_{i}|^{2}dvol_{\tilde{g}|_{U}}$$
	and does not depend either on the choose of the orthonormal basis of $V$ or on the choice of the fundamental domain of the action of $\Gamma$ on $\tilde{X}$. The von Neumann dimension of a $\Gamma$-module $V$ is therefore defined as
	$$\dim_{\Gamma}(V)=\sum_{i\in\N}\int_{U}|\eta_{i}|^{2}dvol_{\tilde{g}|_{U}},$$
	where $\{\eta_{i} \}_{i\in\N}$ is any orthonormal basis for $V$ and $U$ is any fundamental domain of the action of $\Gamma$ on $\tilde{X}$. Since the Laplacian $\De_{\bar{\pa}_{\tilde{E}}}$
	commutes with the action of $\Gamma$, a natural and important example of $\Gamma$-module is provided by the space of $L^{2}$ $\bar{\pa}_{\tilde{E}}$-harmonic forms of bidegree $(i,j)$, $\mathcal{H}^{i,j}_{(2);\bar{\pa}_{\tilde{E}}}(\tilde{X},\tilde{E})$ for each $i,j=0,\cdots,n$ (see \cite[Section 3.6.2]{MM}). We  denote by $\dim_{\Gamma}\mathcal{H}^{p,q}_{(2);\bar{\pa}_{\tilde{E}}}(\tilde{X},\tilde{E})$ the Von Neumann dimension of $\mathcal{H}^{p,q}_{(2);\bar{\pa}_{\tilde{E}}}(\tilde{X},\tilde{E})$ with respect to $\Gamma$, which is a nonnegative real number. We have the following two basic facts. 
	\begin{lemma}
		$$\dim_{\Ga}\mathcal{H}_{(2)}^{k}(M)=0 \Leftrightarrow \mathcal{H}_{(2)}^{k}(M)=\{0\},$$\\
		and $\dim_{\Ga}\mathcal{H}$ is additive: Given $$0\rightarrow\mathcal{H}_{1}\rightarrow\mathcal{H}_{2}\rightarrow \mathcal{H}_{3}\rightarrow 0,$$
		one have $$\dim_{\Ga}\mathcal{H}_{2}=\dim_{\Ga}\mathcal{H}_{1}+\dim_{\Ga}\mathcal{H}_{3}.$$
	\end{lemma}
	Then the $L^{2}$-Hodge numbers of $(X,E)$, denote by $h^{p,q}_{(2)}(X,E)$, are defined to be
	$$h^{p,q}_{(2)}(X,E)=\dim_{\Gamma}\mathcal{H}^{p,q}_{(2);\bar{\pa}_{\tilde{E}}}(\tilde{X},\tilde{E})\in\mathbb{R}^{\geq0},\ (0\leq p,q\leq n).$$
	The Dolbeault-type operators $\mathcal{D}_{p}$ can be lifted to $(\tilde{X},\tilde{E})$:
	$$\tilde{\mathcal{D}}_{p}:\Om^{p,\ast}_{(2)}(\tilde{X},\tilde{E})\rightarrow\Om^{p,\ast\pm1}_{(2)}(\tilde{X},\tilde{E}),$$
	and one can define the $L^{2}$-index of the lifted operators $\tilde{\mathcal{D}}_{p}$ by
	\begin{equation*}
	\begin{split}
	L^{2}Index_{\Gamma}(\tilde{\mathcal{D}}_{p})&=\dim_{\Gamma}(\ker\tilde{\mathcal{D}}_{p})-\dim_{\Gamma}(\rm{coker}\tilde{\mathcal{D}}_{p})\\
	&=\dim_{\Gamma}(\oplus_{q\ even}\mathcal{H}^{p,q}_{(2);\bar{\pa}_{\tilde{E}}}(\tilde{X},\tilde{E}))-\dim_{\Gamma}(\oplus_{q\ odd}\mathcal{H}^{p,q}_{(2);\bar{\pa}_{\tilde{E}}}(\tilde{X},\tilde{E}))\\
	&=\sum_{q\ even}\dim_{\Gamma}\mathcal{H}^{p,q}_{(2);\bar{\pa}_{\tilde{E}}}(\tilde{X},\tilde{E})-\sum_{q\ odd}\dim_{\Gamma}\mathcal{H}^{p,q}_{(2);\bar{\pa}_{\tilde{E}}}(\tilde{X},\tilde{E})\\
	&=\sum_{q=0}^{n}(-1)^{q}h_{(2)}^{p,q}(X,E).\\
	\end{split}
	\end{equation*}
	We recall the Atiyah's $L^{2}$-index theorem \cite{Ati,Pan}.
	\begin{theorem}\cite[Theorem 6.1]{Pan}
		Let $X$ be closed Riemannian manifold, $P$ a determined elliptic operator on sections of certain bundles over $X$. Denote  by $\tilde{P}$ its lift of $P$ to the universal convering space $\tilde{X}$. Let $\Ga=\pi_{1}(M)$. Then   the $L^{2}$ kernel of $\tilde{P}$ has a finite $\Ga$-dimension and 
		$$L^{2}Index_{\Ga}(\tilde{P})=Index(P).$$
	\end{theorem}
	We  define $L^{2}$- Euler characteristics 
	$$\chi^{p}_{(2)}(X,E)=\sum_{q=0}^{n}(-1)^{q}h^{p,q}_{(2)}(X,E)$$
	on the holomorphic bundle $\Om^{p,0}(X)\otimes E$ over a compact K\"{a}hler manifold. The celebrated $L^{2}$-index theorem of Atiyah \cite{Ati} asserts that
	$$Index(\mathcal{D}_{p})=L^{2}Index_{\Gamma}(\tilde{\mathcal{D}}_{p})$$
	so we have the following crucial identities between $\chi^{p}(X,E)$ and the $L^{2}$-Hodge numbers $h^{p,q}_{(2)}(X,E)$:
	$$\chi^{p}(X,E)=\chi^{p}_{(2)}(X,E)=\sum_{q=0}^{n}(-1)^{q}h^{p,q}_{(2)}(X,E).$$
	\begin{theorem}\label{T9}
		Let $(X,\w)$ be a compact K\"{a}hler manifold, $E$ a holomorphic bundle on $X$. Then  for any $p=0,\cdots,n$,
		$$\chi^{p}(X,E)=\chi^{p}_{(2)}(X,E)=\sum_{q=0}^{n}(-1)^{q}h^{p,q}_{(2)}(X,E).$$	
	\end{theorem}
	In \cite{Gro}, Gromov proved that if $X$ is K\"{a}hler hyperbolic, $\dim_{\C}X=n$, then for every $p=0,1,\ldots,n$, the Euler characteristic $$\chi^{p}(X)=\int_{X}td(X)ch(\Om^{p,0}(TX))$$
	does not vanish and $sign\chi^{p}=(-1)^{n-p}$. We will extend the result to holomorphic vector bundle case.
	\begin{proposition}\label{P4}
		Let $(X,\w)$ be a compact K\"{a}hler manifold with sectional curvature bounded from above by a negative constant, i.e.,
		$$sec\leq -K,$$ 
		for some $K>0$. Let $E$ be a holomorphic vector bundle on $X$, $D_{E}$ be the  Chern connection on $E$. If the curvature $\Theta(E)$ of $D_{E}$ such that 
		$$|[\La,i\Theta(E)]|\leq c(n)K,$$
		then  
		\begin{equation*}
		\left\{
		\begin{aligned}
		h^{p,q}_{(2)}(X,E)=0,\  & p+q\neq n\\
		h^{p,q}_{(2)}(X,E)\geq 1,\  & p+q=n\\
		\end{aligned}  
		\right.
		\end{equation*}
		In particular, for every $p=0,1,\cdots,n$, the Euler characteristic 
		$$(-1)^{n-p}\chi^{p}(X,E)\geq 1.$$
	\end{proposition}
	\begin{proof}
		Following Theorems \ref{T3.12} and \ref{T9}, we have
		$$(-1)^{n-p}\chi^{p}(X,E)=h^{p,n-p}_{(2)}(X,E)\geq1.$$
	\end{proof}
	If we denote by $K_{j}(M,E)$ $(0\leq j\leq n)$ the coefficients in the Taylor expansion of $\chi_{y}(X,E)$ at $y=-1$, i.e., $$\chi_{y}(X,E):=\sum_{j=0}^{n}K_{j}(X,E)\cdot(y+1)^{j}.$$
	Following the idea of Li in \cite{Li}, we then have
	\begin{theorem}\label{T8}
		Let $(X,\w)$ be a  compact K\"{a}hler manifold with sectional curvature bounded from above by a negative constant, i.e.,
		$$sec\leq -K,$$ 
		for some $K>0$. Let $E$ be a holomorphic vector bundle on $X$, $D_{E}$ be the  Chern connection on $E$. If the curvature $\Theta(E)$ of  $D_{E}$ such that 
		$$C:=|[\La,i\Theta(E)]|\leq c(n)K,$$
		then for all $0\leq j\leq n$, $(X,E)$ satisfy Chern number inequalities
		\begin{equation*}
		(-1)^{n+j}K_{j}(X,E)\geq\sum_{p=j}^{n}\tbinom{n}{p}.
		\end{equation*}
	\end{theorem} 
	\begin{proof}
		Following Proposition \ref{P4}, we get
		$$\chi^{p}(X,E)=\sum_{q=0}^{n}(-1)^{q}h^{p,q}(X,E)=(-1)^{n-p}h_{(2)}^{p,n-p}(X,E).$$
		We have
		\begin{equation}\label{E1}
		\begin{split}
		(-1)^{n}\sum_{j=0}^{n}K_{j}(X,E)\cdot(y+1)^{n}&=(-1)^{n}\chi_{y}(X,E)\\
		&=(-1)^{n}\sum_{p=0}^{n}\chi^{p}(X,E)\cdot y^{p}\\
		&=\sum_{p=0}^{n}h^{p,n-p}(X,E)\cdot(-y)^{p}.\\
		\end{split}
		\end{equation}
		Now comparing the coefficients of the Taylor expansion at $y=-1$ on both sides of (\ref{E1}) yields
		\begin{equation*}
		\begin{split}
		(-1)^{n}K_{j}(X,E)&=\frac{1}{j!}[\sum_{p=0}^{n}h_{(2)}^{p,n-p}(X,E)\cdot(-y)^{p}]^{(j)}|_{y=-1}\\
		&=(-1)^{j}\sum_{p=j}^{n}\tbinom{p}{j}h^{p,n-p}_{(2)}(X,E).
		\end{split}
		\end{equation*}
		This implies that
		$$(-1)^{n+j}K_{j}(X,E)\geq\sum_{p=j}^{n}\tbinom{p}{j}.$$
	\end{proof}
	\begin{remark}
		If $E$ is flat bundle, then
		$$\chi_{y}(X,E)=rank(E)\chi_{y}(X)=\sum_{j=0}^{n}K_{j}(X)\cdot(y+1)^{j}.$$
		Therefore,
		$$K_{j}(X,E)=rank(E)K_{j}(X).$$
		The first few terms are given by
		$$K_{0}(X)=c_{n}[X],\ K_{1}(X)=-\frac{1}{2}nc_{n}[X],\cdots.$$
		A recursive algorithm for calculating $K_{j}$ was described in \cite[p. 144]{LW}. The formulas $K_{j}$ for $j\leq6$ are presented, respectively, in \cite[pp. 141--143]{LW}, \cite[p. 145]{Sa}.  \\
		If $E$ is a holomorphic bundle, then the first few terms of $K_{j}(X,E)$ are given by
		$$K_{0}(X,E)=\chi_{y}(X,E)|_{y=-1}=rank(E)c_{n}[X],$$ 
		\begin{equation*}
		\begin{split}
		K_{1}(X,E)&=\frac{d}{dy}\chi_{y}(X,E)|_{y=-1}\\
		&=\frac{d}{dy}\int_{X}ch(E)\prod_{i=1}^{n}(1+ye^{-\gamma_{i}})\frac{\gamma_{i}}{1-e^{-\gamma_{i}}}|_{y=-1}\\
		&=\int_{X}ch(E)td(X)\sum_{i=1}^{n}(\frac{e^{-\gamma_{i}}}{1-e^{-\gamma_{i}}}\prod_{i=1}^{n}(1-e^{-\gamma_{i}}))\\
		&=\int_{X}ch(E)\prod_{i=1}^{n}\gamma_{i}\sum_{i=1}^{n}(\frac{e^{-\gamma_{i}}}{1-e^{-\gamma_{i}}})\\ 
		&=\int_{X}ch(E)\prod_{i=1}^{n}\gamma_{i}\sum_{i=1}^{n}(-1+\frac{1}{1-e^{-\gamma_{i}}})\\ 
		&=\int_{X}ch(E)\sum_{i=1}^{n}\big{(}(\prod_{j\neq i}\gamma_{j})(\frac{\gamma_{i}}{1-e^{-\gamma_{i}}}-\gamma_{i})\big{)}\\ 
		&=\int_{X}ch(E)\sum_{i=1}^{n}\big{(}(\prod_{j\neq i}\gamma_{j})(1-\frac{\gamma_{i}}{2})\big{)}\\
		&=\int_{X}(rank(E)+c_{1}(E))(c_{n-1}(X)-\frac{n}{2}c_{n}(X))\\
		&=-\frac{rank(E)}{2}nc_{n}[X]+\langle c_{n-1}(X)c_{1}(E),X\rangle.\\
		\end{split}
		\end{equation*}
		Here we use the fact
		$$c_{n-1}(X)=\sum_{i=1}^{n}(\prod_{j\neq i}\gamma_{j}),\ c_{n}(X)=\prod_{i=1 }^{n}\gamma_{i}.$$
	\end{remark}
	\begin{corollary}\label{C3}
		Let $(X,\w)$ be a  compact K\"{a}hler surface with sectional curvature bounded from above by a negative constant, i.e.,
		$$sec\leq -K,$$ 
		for some $K>0$. Let $E$ be a holomorphic vector bundle on $X$, $D_{E}$ be the  Chern connection on $E$. If the curvature $\Theta(E)$ of  $D_{E}$ such that 
		$$C:=|[\La,i\Theta(E)]|\leq c(n)K,$$
		we then have
		\begin{equation*}
		\int_{X}c_{1}(X)c_{1}(E)=0.
		\end{equation*}
		Furthermore, if $ch_{2}(E)=0$, then
		$$\chi_{y}(X,E)=rank(E)\chi_{y}(X).$$
	\end{corollary} 
	\begin{proof}
		For any $\a\in\Om^{2,0}(X,E)$ or $\a\in\Om^{0,2}(X,E)$, we observe that
		$$[\La,i\Theta(E)]\a=0.$$
		Therefore, 
		$$\De_{E}\a=2\De_{\bar{\pa}_{E}}\a=2\De_{\pa_{E}}\a.$$
		Hence
		$$\ker\De_{\bar{\pa}_{E}}\cap\Om^{0,2}_{(2)}(X,E)=\ker\De_{E}\cap\Om^{0,2}_{(2)}(X,E)$$
		and
		$$\ker\De_{\bar{\pa}_{E}}\cap\Om^{2,0}_{(2)}(X,E)=\ker\De_{E}\cap\Om^{2,0}_{(2)}(X,E).$$
		Noticing that $\Om^{0,2}(X,E)\cong \Om^{2,0}(X,E)$. We then have
		$$\ker\De_{\bar{\pa}_{E}}\cap\Om^{2,0}_{(2)}(X,E)\cong\ker\De_{\bar{\pa}_{E}}\cap\Om^{0,2}_{(2)}(X,E).$$
		One can see that $$h^{2,0}_{(2)}(X,E)=h^{0,2}_{(2)}(X,E).$$
		Noticing that $$\chi_{y}(X,E)|_{y=-1}=h^{0,2}_{(2)}(X,E)+h_{(2)}^{1,1}(X,E)+h^{2,0}_{(2)}(X,E).$$
		We then have
		\begin{equation*}
		\begin{split}
		-K_{1}(X,E)&=rank(E)\int_{X}c_{2}(X)-\int_{X}c_{1}(X)c_{1}(E)\\
		&=h_{(2)}^{1,1}(X,E)+2h^{2,0}_{(2)}(X,E)\\
		&=h^{0,2}_{(2)}(X,E)+h_{(2)}^{1,1}(X,E)+h^{2,0}_{(2)}(X,E)\\
		&=\chi_{y}(X,E)|_{y=-1}\\
		&=rank(E)\chi(X).\\
		\end{split}
		\end{equation*}
		Therefore, we get
		\begin{equation*}
		\int_{X}c_{1}(X)c_{1}(E)=0.
		\end{equation*}
		Noticing that
		$$td(X)=1+\frac{c_{1}(X)}{2}+\frac{c^{2}_{1}(X)+c_{2}(X)}{12},\  ch(E)=rank(E)+c_{1}(E)+\frac{c^{2}_{1}(E)-2c_{2}(E)}{2}.$$
		By the definition of $K_{2}(X,E)$, we then have
		\begin{equation*}
		\begin{split}
		K_{2}(X,E)&=\frac{1}{2}\frac{d^{2}}{dy^{2}}\chi_{y}(X,E)|_{y=-1}\\
		&=\frac{1}{2}\frac{d^{2}}{dy^{2}}\int_{X}td(X)ch(E)\prod_{i=1}^{2}(1+ye^{-\gamma_{i}})|_{y=-1}\\
		&=\frac{1}{2}\frac{d^{2}}{dy^{2}}\int_{X}td(X)ch(E)e^{-\gamma_{1}-\gamma_{2}}y^{2}|_{y=-1}\\
		&=\int_{X}td(X)ch(E)e^{-c_{1}(X)})\\
		&=\int_{X}(1+\frac{c_{1}(X)}{2}+\frac{c^{2}_{1}(X)+c_{2}(X)}{12})(1-c_{1}(X)+\frac{c^{2}_{1}(X) }{2})(rank(E)+c_{1}(E)+\frac{c^{2}_{1}(E)-2c_{2}(E)}{2})\\
		&=\int_{X}\big{(}\frac{c^{2}_{1}(E)-2c_{2}(E)}{2}-c_{1}(X)c_{1}(E)+rank(E)\frac{c^{2}_{1}(X)}{2}-rank(E)\frac{c^{2}_{1}(X)}{2}+\frac{c_{1}(X)}{2}c_{1}(E)\\
		&+rank(E)\frac{c^{2}_{1}(X)+c_{2}(X)}{12}\big{)}\\
		&=rank(E)K_{2}(X)-\langle\frac{c_{1}(X)c_{1}(E)}{2},X\rangle+\langle\frac{c^{2}_{1}(E)-2c_{2}(E)}{2},X\rangle.\\
		\end{split}
		\end{equation*}
		We also have
		\begin{equation*}
		K_{2}(X,E)=h^{2,0}_{(2)}(X,E),\ K_{2}(X)=h^{2,0}_{(2)}(X).
		\end{equation*}
		Therefore, we get
		\begin{equation*}
		\int_{X}\frac{c^{2}_{1}(E)-2c_{2}(E)}{2}=h^{2,0}_{(2)}(X,E)-rank(E)h^{2,0}_{(2)}(X).
		\end{equation*}
		If $ch_{2}(E)=0$, then 
		$$K_{2}(X,E)=rank(E)K_{2}(X).$$
		Notice that $K_{0}(X,E)=rank(E)K_{0}(X)$ and $K_{1}(X,E)=rank(E)K_{1}(X)$. Therefore, 
		$$\chi_{y}(X,E)=rank(E)\chi_{y}(X).$$
	\end{proof}
	\begin{proof}[\textbf{Proof of Theorem \ref{T2}}]
		If $C:=|[\La,i\Theta(L)]|>c(n)K$, i.e., $\lfloor\frac{c_{n}K}{C}\rfloor=0$, then 	$\chi(X)\geq 3$.\\
		If $C\leq c_{n}K$, i.e., there is a positive integer $N$ such that
		$$C(N+1)>c_{n}K\geq CN=|[\La,i\Theta(L^{\otimes N})]|,$$
		then for any $|m|\leq N$, following Corollary \ref{C3}, we get 
		$$\langle c_{1}(L^{\otimes m})c_{1}(X),X\rangle=0.$$
		Therefore,
		$$K_{1}(X,L^{\otimes m})=K_{1}(X)=-c_{2}[X].$$
		By the definition of $\chi_{y}(X,E)$,  we get
		\begin{equation*}
		\begin{split}
		\chi_{y}(X,L^{\otimes m})&=\chi^{0}(X,L^{\otimes m})+\chi^{1}(X,L^{\otimes m})y+\chi^{2}(X,L^{\otimes m})y^{2}\\
		&=K_{0}(X,L^{\otimes m})+K_{1}(X,L^{\otimes m})\cdot(y+1)+K_{2}(X,L^{\otimes m})\cdot(y+1)^{2}.\\
		\end{split}
		\end{equation*}
		Therefore, for any $|m|\leq N$, we have
		\begin{equation*}
		\begin{split}
		\chi^{0}(X,L^{\otimes m})&=K_{0}(X,L^{\otimes m})+K_{1}(X,L^{\otimes m})+K_{2}(X,L^{\otimes m})\\
		&=K_{0}(X)+K_{1}(X)+K_{2}(X)+\frac{m^{2}}{2}\int_{X}c^{2}_{1}(L)\\
		&=\chi^{0}(X)+\frac{m^{2}}{2}\int_{X}c^{2}_{1}(L)\geq 1,\\
		\end{split}
		\end{equation*}
		\begin{equation*}
		\begin{split}
		\chi^{1}(X,L^{\otimes m})&=K_{1}(X,L^{\otimes m})+2K_{2}(X,L^{\otimes m})\\
		&=K_{1}(X)+2K_{2}(X)+m^{2}\int_{X}c^{2}_{1}(L)\\
		&=\chi^{1}(X)+m^{2}\int_{X}c^{2}_{1}(L)\leq-1,\\
		\end{split}
		\end{equation*}
		\begin{equation*}
		\begin{split}
		\chi^{2}(X,L^{\otimes m})&=K_{2}(X,L^{\otimes m})=K_{2}(X)+\frac{m^{2}}{2}\int_{X}c^{2}_{1}(L)\\
		&=\chi^{2}(X)+\frac{m^{2}}{2}\int_{X}c^{2}_{1}(L)\geq 1.\\
		\end{split}
		\end{equation*}
		Following \cite[0.4.A. Theorem]{Gro}, we also have $(-1)^{p}\chi^{p}(X)\geq1$, $p=0,1,2$.\\
		When $\int_{X}c^{2}_{1}(L)>0$, for any $|m|\leq N$, we obtain that
		\begin{equation*}
		\begin{split}
		&\chi^{0}(X,L^{\otimes m})\geq 1+\frac{m^{2}}{2}\int_{X}c^{2}_{1}(L),\\
		&-\chi^{1}(X,L^{\otimes m})\geq 1,\\
		&\chi^{2}(X,L^{\otimes m})\geq 1+\frac{m^{2}}{2}\int_{X}c^{2}_{1}(L).\\
		\end{split}
		\end{equation*}
		Therefore,
		\begin{equation*}
		\begin{split}
		\chi(X)&=\chi_{y}(X,L^{\otimes m})|_{y=-1}\\
		&=\sum_{i=0}^{2}(-1)^{i}\chi^{i}(X,L^{\otimes m})\\
		&\geq 3+m^{2}\int_{X}c^{2}_{1}(L).\\
		\end{split}
		\end{equation*}
		When $\int_{X}c^{2}_{1}(L)<0$, for any $|m|\leq N$, we obtain that
		\begin{equation*}
		\begin{split}
		&\chi^{0}(X,L^{\otimes m})\geq 1,\\
		&-\chi^{1}(X,L^{\otimes m})\geq 1-m^{2}\int_{X}c^{2}_{1}(L),\\
		&\chi^{2}(X,L^{\otimes m})\geq 1.\\
		\end{split}
		\end{equation*}
		Therefore,
		\begin{equation*}
		\begin{split}
		\chi(X)&=\chi_{y}(X,L^{\otimes m})|_{y=-1}\\
		&=\sum_{i=0}^{2}(-1)^{i}\chi^{i}(X,L^{\otimes m})\\
		&\geq 3-m^{2}\int_{X}c^{2}_{1}(L).\\
		\end{split}
		\end{equation*}
		Therefore, for all cases, we get 
		\begin{equation*}
		\chi(X)\geq 3+|\int_{X}c^{2}_{1}(L)|\cdot(\lfloor\frac{c_{n}K}{C}\rfloor)^{2}.
		\end{equation*}
	\end{proof}
	
	\section{Eigenvalue and Euler characteristic}
	\subsection{Hilbert polynomial of line bundle }
	The Hilbert polynomial of polarized manifold $(X,L)$, i.e., $L$ is an ample line bundle on  a compact K\"{a}hler manifold $X$, is defined as the functional $P_{(X,L)}(m):=\chi(X,L^{\otimes L})$. Indeed, $P_{(X,L)}(m)$ is a polynomial in $n$. Following Kodaira vanishing theorem (see \cite[Proposition 5.27]{Huy}), we also have $P_{(X,L)}(m)=h^{0}(X,L^{\otimes m})$ for $m\gg 1$. Notice that
	\begin{equation*}
	\begin{split}
	[td(X)ch(\Om^{p,0}(TX))ch(L^{\otimes m})]_{2n}
	&=\sum_{i=0}^{n}[td(X)ch(\Om^{p,0}(TX))]_{2n-2i}[ch^{m}(L)]_{2i},\\
	&=\sum_{i=0}^{n}[td(X)ch(\Om^{p,0}(TX))]_{2n-2i}\frac{(mc_{1}(L))^{i}}{i!},\\
	\end{split}
	\end{equation*}
	where $[td(X)ch(\Om^{p,0}(TX))]_{2n-2i}$ is the part of $2n-2i$-form of $td(X)ch(\Om^{p,0}(TX)).$
	We denote 	$$a_{i}:=\int_{X}[td(X)ch(\Om^{p,0}(TX))]_{2n-2i}\wedge\frac{c^{i}_{1}(L)}{i!}.$$
	Therefore, 
	\begin{equation*}
	P^{(p)}_{n}(m,L):=\sum_{i=0}^{n}a_{i}m^{i}=\int_{X}td(X)ch(\Om^{p,0}(TX))ch(L^{\otimes m}).
	\end{equation*}
	$P^{(p)}_{n}(m,L)$ is a polynomial of degree $n$ if only if $\int_{X}c^{n}_{1}(L)\neq 0$. We can introduce the following definition.
	\begin{definition}
		Let $L$ be a holomorphic line bundle on a compact K\"{a}hler manifold $X$. We call $$P^{(p)}_{n}(m,L):=\chi^{p}(X,L^{\otimes m})$$ 
		the $p$-Hilbert polynomial of line bundle $L$.
	\end{definition}
	\begin{lemma}\label{L4.2}
		Let $P_{n}(m)$ be a numerical polynomial of degree $n\geq1$. Suppose that $P_{n}(m)$ is not constant. If $N=\{i_{1},\cdots,i_{2nL+1}\}$, where the integers $\{i_{j}\}_{j=1,\cdots,2nL+1}$ that are not equal to each other, then there exists integer $\tilde{i}$ such that 
		$$|P_{n}(\tilde{i})|\geq L.$$ 
	\end{lemma}
	\begin{proof}
		Since $P_{n}(m)$ is not constant, there is an coefficient $a_{i}\neq 0$. Notice that for any integer $i$, the equation $P_{n}(x)=i$ has at most $n$ solutions. Therefore, $0\leq |P_{n}(\tilde{i})|\leq L-1$ has at most $2nL$ solutions. Hence there exists a integer $\tilde{i}\in N$ such that $|P_{n}(\tilde{i})|\geq L$.
	\end{proof}
	Assume that $X$ has a K\"{a}hler metric $\w$. Let
	$$\gamma_{1}(x)\leq\cdots\gamma_{n}(x)$$
	be the eigenvalues of $i\Theta(L)_{x}$ with respect to $\w_{x}$ at each point $x\in X$, and let 
	$$i\Theta(L)=i\sum_{1\leq j\leq n}\gamma_{i}\xi_{j}\wedge\bar{\xi}_{j},\ \xi_{j}\in T^{\ast}_{x}X$$
	be a diagonalization of $i\Theta(L)_{x}$. We then have
	\begin{equation*}
	\begin{split}
	\langle [i\Theta(L),\La]u,u\rangle=\sum_{J,K}(\sum_{j\in J}\gamma_{j}+\sum_{j\in K}\gamma_{j}-\sum_{1\leq j\leq n}\gamma_{j})|u_{J,K}|^{2}.
	\end{split}
	\end{equation*}
	\begin{lemma}
		Let $L$ be a line on a compact K\"{a}hler manifold $(X,\w)$. Then $C=0$ if only if $\Theta(L)=0$. Furthermore, if $L$ is not flat, then there is an uniform positive constant $\varepsilon\in(0,1)$ such that 
		$$C>\varepsilon.$$
	\end{lemma}
	\begin{proof}
		If $C=0$, then for any $J,K$, we have
		$$\sum_{J,K}(\sum_{j\in J}\gamma_{j}+\sum_{j\in K}\gamma_{j}-\sum_{1\leq j\leq n}\gamma_{j})=0.$$
		We take $J=\{1,\cdots,n\}$ and $K=j$, then we get
		$$\gamma_{j}=0, \forall\ 0\leq j\leq n,$$
		i.e., $\Theta(L)=0$.\\
		We suppose that the constant $\varepsilon$ does not exist. We may then choose a sequence non-flat bundles $\{L_{i}\}_{i\in\N}$ such that $C_{i}\rightarrow0$ as $i\rightarrow\infty$. We have
		$$|\Theta(L_{i})|\leq\sum_{j=1}^{n}|\gamma_{j}|\leq nC_{i}.$$
		If $i$ large enough, then there exists a flat connection on $L_{i}$ (see \cite{Uhl}), contradicting our initial assumption regarding the sequence $\{L_{i}\}_{i\in N}$.
	\end{proof}
	
	\begin{proof}[\textbf{Proof of Theorem \ref{T4}}]
		Since $\chi^{p}(X,L^{\otimes m})$ is not constant for some $p\in[0,n]$, we obtain that $L$ is not flat, i.e.,  $$C:=|[i\Theta(L),\La]|>0.$$
		For any $K>0$,  there is an integer $N$ such that
		$$C(nN+1)>c_{n}K\geq C(nN)=|[i\Theta(L^{\otimes (nN)}),\La]|.$$ 
		Notice that $P^{(p)}_{n}(m,L)$ is integer for any $|m|\leq nN$. Following Lemma \ref{L4.2}, then there is a integer $\tilde{m}=\tilde{m}(p)$ such that
		$$|P^{(p)}_{n}(\tilde{m},L)-\chi^{p}(X)|\geq N.$$
		We then have either
		$$(-1)^{n-p+1}(P^{(p)}_{n}(\tilde{m},L)-\chi^{p}(X))\geq N $$
		or
		$$(-1)^{n-p+1}(P^{(p)}_{n}(\tilde{m},L)-\chi^{p}(X))\leq -N.$$
		If $(-1)^{n-p+1}(P^{(p)}_{n}(\tilde{m},L)-\chi^{p}(X))\leq -N$, then
		$$(-1)^{n-p}\chi^{p}(X,L^{\otimes \tilde{m}})=(-1)^{n-p}\chi^{p}(X)+(-1)^{n-p}(P^{(p)}_{n}(\tilde{m},L)-\chi^{p}(X))\geq1+N.$$
		Following Theorem \ref{T1}, for any $|m|\leq nN$, we get
		$$(-1)^{n-p}\chi^{p}(X,L^{\otimes m})=(-1)^{n-p}\chi^{p}(X)+(-1)^{n-p}(P^{(p)}_{n}(\tilde{m},L)-\chi^{p}(X))\geq1.$$
		If $(-1)^{n-p+1}(P^{(p)}_{n}(\tilde{m},L)-\chi^{p}(X))\geq N$, then there is a integer $\tilde{m}$ such that
		$$(-1)^{n-p}\chi^{p}(X)=(-1)^{n-p}\chi^{p}(X,L^{\otimes m})+(-1)^{n-p+1}(P^{(p)}_{n}(\tilde{m},L)-\chi^{p}(X))\geq1+N.$$
		Following the vanishing theorem and the Atiyah's $L^{2}$-index theorem, we obtain that  either  
		$$(-1)^{n-p}\chi^{p}(X)=(-1)^{n-p}\chi^{p}_{(2)}(X)=h^{n-p,p}_{(2)}(X)\geq 1+N$$
		or
		$$(-1)^{n-p}\chi^{p}(X,L^{\otimes\tilde{m}})=(-1)^{n-p}\chi^{p}_{(2)}(X,L^{\otimes\tilde{m}})=h^{n-p,p}_{(2)}(X,L^{\otimes\tilde{m}})\geq 1+N.$$
		Therefore, we get
		$$(-1)^{n}\chi(X)=(-1)^{n}\chi_{(2)}(X)=\sum_{p=0}^{n}h^{n-p,p}_{(2)}(X)\geq (n+1)+N$$
		or
		$$(-1)^{n}\chi(X,L^{\otimes\tilde{m}})=(-1)^{n}\chi^{p}_{(2)}(X,L^{\otimes\tilde{m}})=\sum_{p=0}^{n}h^{n-p,p}_{(2)}(X,L^{\otimes\tilde{m}})\geq (n+1)+N.$$
		We observe that for any $m\in\mathbb{Z}$, 
		$$\chi_{y}(X,L^{\otimes m})|_{y=-1}=\chi(X).$$
		Hence, the Euler number of $X$ must satisfy
		$$(-1)^{n}\chi(X)=(-1)^{n}\chi(X,L^{\otimes\tilde{m}})\geq (n+1)+N.$$
	\end{proof}
	
	\subsection{The roots of Hilbert polynomial}
	We denote by $$Z^{\pm}=\{m\in\mathbb{R^{\pm}}: P^{(p)}_{n}(m,L)=\chi^{p}(X)\}$$
	the set of positive (resp. negative) roots of $(P^{(p)}_{n}(m,L)-\chi^{p}(X))$.
	We denote  
	$$C^{\pm}:=\max_{m\in Z^{\pm}}|m|,$$
	(when $Z^{\pm}=\emptyset$, we denote $C^{\pm}=0$). It's easy to see that $C^{\pm}$ depends on $K$ and $c_{1}(L)$.
	\begin{lemma}(\cite[Lemma 16.3]{Dem96})\label{L4.4}
		Let $P_{n}(m)$ be a numerical polynomial of degree $n\geq1$ and with leading coefficient $\frac{1}{n!}a_{n}\in\mathbb{Z}$, $a_{n}>0$. We assume that $P_{n}(m)\geq 0$ for all $m\geq m_{0}$. Then for any $k\in\mathbb{N}$, there exists $m\in[m_{0},m_{0}+kn]$ such that 
		$$P_{n}(m)\geq \frac{a_{n}k^{n}}{2^{n-1}}.$$	
	\end{lemma}
	\begin{proof}
		By virtue of Newton's formula for the iterated differentials $\De P_{n}(m)=P_{n}(m+1)-P_{n}(m)$, we obtain 
		$$\De^{n}P_{n}(m)=\sum_{1\leq j\leq n}(-1)^{j}\tbinom{n}{j}P_{n}(m+N-j)=a_{n}.$$
		Consequently, if $j\in\{0,2,4,2\lfloor n/2\rfloor\}\subset[0,n]$ is the even integer realizing the maximum of $P(m_{0}+n-j)$ on this set, we obtain
		$$2^{n-1}P(m_{0}+n-j)\geq(\tbinom{n}{0}+\tbinom{n}{2}+\cdots)P(m_{0}+n-j)\geq a_{n},$$
		whereby we obtain the existence of an integer $m\in[m_{0},m_{0}+n]$ with $P_{n}(m)\geq \frac{a_{n}}{2^{n-1}}$. The result is therefore prove for $k=1$. In general case, we apply this particular result to the polynomial $Q_{n}(m)=P_{n}(km-(k-1)m_{0})$, for which the leading coefficient is $\frac{1}{n!}a_{n}k^{n}$. 
	\end{proof}
	
	\begin{proposition}\label{C1}
		Let $(X,\w)$ be a  compact K\"{a}hler manifold with sectional curvature bounded from above by a negative constant, i.e.,
		$$sec\leq -K,$$ 
		for some $K>0$. Suppose that there is a holomorphic line bundle $L$ on $X$ such that $a_{n}:=\int_{X}c^{n}_{1}(L)\neq 0$. 
		If $c_{n}K\geq CC^{\pm}$, then there exists a integer $\tilde{m}\in[-\lfloor\frac{c_{n}K-CC^{\pm}}{nC}\rfloor, \lfloor\frac{c_{n}K-CC^{\pm}}{nC}\rfloor]$ such that either
		\begin{equation*}
		(-1)^{n-p}\chi^{p}(X,L^{\otimes{\tilde{m}}})\geq 2|a_{n}|(|\lfloor\frac{c_{n}K-CC^{\pm}}{2Cn}\rfloor|)^{n}+1
		\end{equation*}
		or
		\begin{equation*}
		(-1)^{n-p}\chi^{p}(X)\geq 2|a_{n}|(|\lfloor\frac{c_{n}K-CC^{\pm}}{2Cn}\rfloor|)^{n}+1.
		\end{equation*}
	\end{proposition}
	\begin{proof}
		For any $K>0$,  there are integers $N^{\pm}$, $N$  such that
		$$Cn(N+1)> c_{n}K-CC^{\pm}\geq C(nN).$$
		and
		$$ N^{\pm}+1> C^{\pm}\geq N^{\pm}.$$
		Since  $c_{n}K\geq CC^{\pm}$, $N\geq0$. Therefore, we have
		$$c_{n}K\geq C(nN+N^{\pm})\geq|[i\Theta(L^{\otimes (nN+N^{\pm})}),\La]|.$$
		Notice that $sign(a_{n})(P^{(p)}_{n}(m,L)-\chi^{p}(X))$ is positve integer for any $m>N^{+}$ and $sign(a_{n})(P^{(p)}_{n}(m,L)-\chi^{p}(X))$ is negative (resp. positive) integer for any $m<-N^{-}$ when $n$ is odd (resp. even). Following the way in Lemma \ref{L4.4}, then there is a integer $N^{+}\leq \tilde{m}\leq nN+N^{+}$ (resp. $-nN-N^{-}\leq\tilde{m}\leq -N^{-}$) such that
		$$sign(a_{n})(P^{(p)}_{n}(\tilde{m},L)-\chi^{p}(X))\geq |a_{n}|N^{n}/2^{n-1}\ (resp.\ (-1)^{n}sign(a_{n})P_{n}(\tilde{m})\geq|a_{n}|N^{n}/2^{n-1}).$$
		Following Theorem \ref{T1}, for any $|m|\leq (nL+N^{\pm})$, we get
		$$(-1)^{n-p}\chi^{p}(X,L^{\otimes m})=(-1)^{n-p}\chi^{p}(X)+(-1)^{n-p}(P^{(p)}_{n}(m,L)-\chi^{p}(X))\geq1$$
		and
		$$(-1)^{n-p}\chi^{p}(X)=(-1)^{n-p}\chi^{p }(X,L^{\otimes m})+(-1)^{n-p+1}(P^{(p)}_{n}(m,L)-\chi^{p}(X))\geq1.$$	
		We then have 
		\begin{equation*}
		\begin{split}
		\max\{(-1)^{n-p}\chi^{p}(X),(-1)^{n-p}\chi^{p }(X,L^{\otimes\tilde{m}}) \}
		&\geq 1+|P^{(p)}_{n}(\tilde{m},L)-\chi^{p}(X)|\\
		&\geq1+|a_{n}|N^{n}/2^{n-1}\\
		&\geq 2|a_{n}|(|\lfloor\frac{c_{n}K-CC^{\pm}}{2Cn}\rfloor|)^{n}+1.\\
		\end{split}
		\end{equation*}
	\end{proof}
	We denote by $$Z_{p}:=\{m\in\mathbb{R}: P^{(p)}_{n}(m,L)=\chi^{p}(X)\}$$
	the set of real roots of polynomial $P^{(p)}_{n}(m,L)-\chi^{p}(X)$. We denote  
	$$m_{p}(L)=\max_{m\in Z_{p}}|m|,$$
	(when $Z=\emptyset$, we denote $m_{p}(L)=0$). It's also easy to see $m_{p}(L)$ depends on $K$ and $c_{1}(L)$.
	\begin{proof}[\textbf{Proof of Theorem \ref{T5}}]
		There exists a integer $N$ such that
		$$Cn(N+1)>c_{n}K-Cm_{p}(L)\geq CnN.$$
		If $N\leq 0$, then $(-1)^{n-p}\chi^{p}(X)\geq1$.\\
		If $N>0$, we then have
		$$c_{n}K\geq C(nN+m_{p}(L))\geq|[i\Theta(L^{\otimes (nN+\lfloor m_{p}(L)\rfloor)}),\La]|.$$
		Following Theorem \ref{T1}, for any $|m|\leq (nN+\lfloor m_{p}(L)\rfloor)$, we get
		$$(-1)^{n-p}\chi^{p}(X,L^{\otimes m})\geq1,$$
		i.e.,
		$$(-1)^{n-p}\chi^{p}(X)\geq(-1)^{n-p+1}(P^{(p)}_{n}(m,L)-\chi^{p}(X))+1.$$
		When $n$ is odd, we get $sign(a_{n})(P^{(p)}_{n}(m,L)-\chi^{p}(X))$ is positve integer for any $m>m_{p}(L)$ and $sign(a_{n})(P^{(p)}_{n}(m,L)-\chi^{p}(X))$ is negative integer for any $m<-m_{p}(L)$. Therefore, following Lemma \ref{L4.4}, we get
		\begin{equation*}
		\begin{split}
		(-1)^{n-p}\chi^{p}(X)&\geq\max_{nN+\lfloor m_{p}(L)\rfloor)\geq|m|>m_{p}(L)}(-1)^{n-p+1}(P^{(p)}_{n}(m,L)-\chi^{p}(X))+1\\
		&\geq 1+|a_{n}|N^{n}/2^{n-1}.
		\end{split}
		\end{equation*}
		When $n$ is even, we get $sign(a_{n})(P^{(p)}_{n}(m,L)-\chi^{p}(X))$ is positve integer for any $|m|>m_{p}(L)$. Following Proposition \ref{C1}, there exists a integer $\tilde{m}\in[-nN-\lfloor m_{p}(L)\rfloor,nN+\lfloor m_{p}(L)\rfloor]$ such that either
		\begin{equation*}
		(-1)^{n-p}\chi^{p}(X,L^{\otimes{\tilde{m}}})\geq1+|a_{n}|N^{n}/2^{n-1}
		\end{equation*}
		or
		\begin{equation*}
		(-1)^{n-p}\chi^{p}(X)\geq 1+|a_{n}|N^{n}/2^{n-1}.
		\end{equation*}
		Therefore, for all cases, we get  
		$$(-1)^{n}\chi(X)\geq\max\{n+1,n+1+2|a_{n}|sign(\lfloor{c_{n}K-Cm_{p}(L)}\rfloor)(|\lfloor\frac{c_{n}K-Cm_{p}(L)}{2Cn}\rfloor|)^{n} \}.$$ 	
	\end{proof}
	
	\subsection*{Acknowledgements}
	The author thanks Zhongjie Lu, Jijian Song, Ruiran Sun for their useful comments and suggestions. This work was supported in part by NSF of China (11801539) and the Fundamental Research Funds of the Central Universities (WK3470000019), the USTC Research Funds of the Double First-Class Initiative (YD3470002002).

\end{document}